\documentclass[a4, 12pt]{amsart}
\UseRawInputEncoding
\usepackage{mathrsfs}
\usepackage{amsmath}
\usepackage{amsfonts}
\usepackage{amssymb}
\usepackage{amsmath,amssymb,amsthm}
\usepackage{amsmath,amssymb,amsthm,amscd}
\usepackage[frame,cmtip,arrow,matrix,line,graph,curve]{xy}
\usepackage{graphpap, color}
\usepackage[mathscr]{eucal}
\usepackage{color}
\usepackage{verbatim}
\usepackage[colorlinks, linkcolor=red, anchorcolor=green,citecolor=blue]{hyperref}
\usepackage{cite}
\usepackage{slashed}

\numberwithin{equation}{section} \numberwithin{equation}{section}
\newtheorem{thm}{Theorem}[section]
\newtheorem{lem}[{thm}]{Lemma}

\newtheorem{prop}[{thm}]{Proposition}
\newtheorem{corr}[{thm}]{Corollary}
\newtheorem{rem}[{thm}]{Remark}
\newtheorem{con}[{thm}]{Conjecture}

\newtheorem*{ack}{Acknowledgment}
\newtheorem*{Nash}{Nash's Theorem}
\newtheorem*{GPPW}{Genernal Payne-P\'{o}lya-Weinberger Conjecture}
\newtheorem*{Will}{Willmore Conjecture}

\newcommand{\DOI}[1]{doi: \href{https://doi.org/#1}{#1}}

\renewcommand{\oddsidemargin}{5mm}

\makeatletter

\title[Spectrum of the Dirac operator]{Spectrum of the Dirac operator\\ on Compact Riemannian Manifolds}
\author[L. Zeng ]{Lingzhong Zeng}

\address{Lingzhong Zeng
\\  \newline \indent Jiangxi Provincial Center for Applied Mathematics$^{1}$
\\  \newline \indent   Jiangxi Normal University, Nanchang 330022,  China.
\\  \newline \indent School of Mathematics and Statistics$^{2}$
\\  \newline \indent   Jiangxi Normal University, Nanchang 330022,  China.
\\ \newline \indent E-mail: lingzhongzeng@yeah.net}

\begin{document}
\maketitle

\begin{abstract}    In this paper, we consider the eigenvalue problem of Dirac operator on a compact Riemannian
manifold isometrically immersed into Euclidean space and derive some extrinsic estimates for the sum of arbitrary consecutive $n$ eigenvalues of the square of
the Dirac operator acting on some Dirac invariant subbundles. As some applications, we deduce some eigenvalue inequalities on the compact submanifolds immersed into Euclidean space, unit sphere or projective spaces and further get some bounds of general Reilly type. In addition, we also establish some universal bounds under certain curvature condition and on the meanwhile provide an alternative proof for Anghel's result. In particular, utilizing  Atiyah-Singer index theorem, we drive an upper bound estimate for the sum of the first $n$ nontrivial eigenvalues of Atiyah-Singer Laplacian acting on the spin manifolds without dimensional assumption.
\end{abstract}


\section{Introduction}

Let $M^{n}$ be an $n$-dimensional compact Riemannian manifold with smooth metric $g$, $\slashed{S} \rightarrow M^{n}$ be a Dirac bundle over $M^{n}$ and  $\slashed{D}$ is a Dirac operator associated with $\slashed{S}$. A fundamental theme in term of Dira operator is to ask how
much information is needed to investigate its spectrum and such a question has at least
peripheral physical meanings. In particular,  when coupled to gauge fields, the lowest
eigenvalue is closely linked to chiral symmetry breaking. From pure geometric perspective, some estimates for the lower eigenvalue may help to give a sharper estimate for
the ADM mass of an asymptotically flat spacetime with black holes or understand the geometric and topological structures, for example, see \cite{An,Ba2,Wit} and references therein. Therefore, eigenvalues of Dirac operators on compact Riemannian manifolds are extensively studied.
On one hand, for the case of lower bound, Friedrich \cite{Fri1} estimated the first eigenvalue of Dirac operator $\slashed{D}$  in the early 1980s:
\begin{equation}\label{eq-fri1}
\Gamma^2(\slashed{D}) \geq \frac{n}{4(n-1)} S_{0},
\end{equation}where  $S_{0}$ is the minimum of scalar curvature of the underling manifold $M^{n}$.
Since then, many intrinsic geometric bounds for the eigenvalues have been established in various settings(see e.g. \cite{Fri2,Gi1} and the references therein). For example, Hijazi \cite{Hij1} showed that
\begin{equation}\label{eq-hij}
\Gamma^2(\slashed{D}) \geq \frac{n}{4(n-1)} \Gamma_1\left(L_{M^{n}}\right)
\end{equation}
for $n \geq 3$, where $L_{M^{n}}=-\frac{4(n-1)}{n-2} \Delta+S_{0}$ is the Yamabe operator of $M^{n}$. For the case of $n=2$, B\"{a}r \cite{Ba1} obtained
\begin{equation}\label{bar-eq-1}
\Gamma^2(\slashed{D}) \geq \frac{4 \pi\left(1-\mathfrak{g}_{0}\right)}{\operatorname{area}(M^{2})},
\end{equation}
where $\mathfrak{g}_{0}$ denotes the genus  of Riemann surface $M^{2}$. We note that those estimates are sharp since all of the equalities in \eqref{eq-fri1}, \eqref{eq-hij} or \eqref{bar-eq-1} give an Einstein metric. Friedrich and Kirchberg \cite{Fri3} proved a lower bound of Dirac operator on non-Ricci flat spin manifold with harmonic curvature tensor and vanishing scalar curvature. However, their lower bound is determined by the minimum of the eigenvalues and the square length of the Ricci tensor.   Using the hypersurface Dirac operator,  Witten \cite{Wit} proved the positive energy conjecture for arbitrary demension. Motivated by Witten's work, Zhang \cite{zhang1,zhang2}  considered the lower bounds for the hypersurface Dirac operators and some related results later are generalized to submanifold Dirac operators by Hijazi and Zhang \cite{HZ1,HZ2}, where they gave optimal lower bounds for the submanifold Dirac operator with respect to the mean curvature and other geometric invaraints, for example, the Yamabe number or the energy momentum tensor under certain extra conditions. Ginoux and Morel \cite{GM} also investigated the eigenvalue problem for submanifold Dirac operators. In addition, Chen and Sun \cite{CS} investigated the eigenvalue of Dirac operators of the twisted bundle  on some submanifolds and gave some conformal lower bounds in terms
of conformal and extrinsic quantities, where they generalized Hijazi's result \cite{Hij1} for $n \geq 3$ and B\"{a}r's result \cite{Ba1} for $n=2$. Recently, Chen \cite{CY} obtained some optimal lower bounds for the eigenvalues of the submanifold Dirac operator on locally reducible Riemannian manifolds in terms of intrinsic and extrinsic expressions.

On the other hand, estimates from above for the eigenvalues of the Dirac operator can be derived in various ways (see, e.g. \cite{An,Ba2,Bu,Che,Fri2}). In particular, based on the variational argument of eigenvalues, Anghel \cite{An} obtained an upper bound on gaps of two consecutive eigenvalues of squared Dirac operator (sometimes we call it Dirac Laplacian) $\slashed{D}^{2}$:
\begin{equation}\label{an-gap}
\Gamma_{k+1} -\Gamma_{k}  \leq n\int_M\left\langle H^{2} \mathfrak{s}_{i}, \mathfrak{s}_{i}\right\rangle dv +\frac{4}{k n} \sum_{i=1}^{k} \Gamma_{i}-\frac{4}{kn} \sum_{l=1}^k\int_{M^{n}}\left\langle\mathfrak{R} \mathfrak{s}_{l}, \mathfrak{s}_{l}\right\rangle dv.
\end{equation}
To convert this into an upper bound on the eigenvalues themselves one has to assume something on the smallest eigenvalue. For example, if $0$ is an eigenvalue and if the scalar curvature vanishes identically, this is, $S_{0} \equiv 0$, then one concludes for the smallest nonzero eigenvalue $\bar{\Gamma}_{1}$\cite[Theorem 3.6]{An}
\begin{equation}\label{An-H}
\bar{\Gamma}_{1}  \leq \frac{n}{\operatorname{vol}(M^{n})} \int_{M}H^{2}dv.
\end{equation}
Let $M^{n}$ be an $n$-dimensional compact Riemannian spin manifold isometrically immersed in $\mathbb{R}^{n+p}$ with second fundamental form $\mathcal{B}$  and $\slashed{D}_{M}^{TN} $ Dirac
operator on $M^{n}$ with coefficients in the spinor bundle of the normal bundle.
Then there is a topologically determined number of Dirac eigenvalues of $M^{n}$ satisfying \cite[Theorem A]{Bau}
\begin{equation}\label{D-eq-sq}
\bar{\Gamma}_{1}^{2}(\slashed{D}_{M}^{TN}) \leq 2^{[n / 2]}\sup_{x\in M^{n}}\|\mathcal{B}\|^2,
\end{equation}where the topologically determined number of Dirac eigenvalues which can be explicitly estimated in terms of codimensions and of characteristic numbers of $M$.
Let $\hat{M}^{n}$, $\tilde{M}^{n}$ and $\bar{M}^{n}$ be three $n$-dimensional closed oriented spin manifolds isometrically immersed in $(n+1)$-dimensional Euclidean space $\mathbb{R}^{n+1}$, unit sphere $\mathbb{S}^{n+1}(1)$ and hyperbolic space $\mathbb{H}^{n+1}(-1)$ with section curvature $-1$,  whose mean curvatures are $\hat{H}$, $\tilde{H}$ and $\bar{H}$, respectively. Then, all of the classical Dirac operators of $\hat{M}^{n}$, $\tilde{M}^{n}$ and $\bar{M}^{n}$ have at least $2^{\left[\frac{n}{2}\right]}$ eigenvalues $\hat{\Gamma}_{1}(\slashed{D}_{\hat{M}}^{TN})$, $\tilde{\Gamma}_{1}(\slashed{D}_{\tilde{M}}^{TN})$ and $\bar{\Gamma}_{1}(\slashed{D}_{\bar{M}}^{TN})$ (counted with multiplicities) satisfying \cite[Corollary 4.2, Corollary 4.3 and Corollary 4.5]{Ba2}
\begin{equation}\label{Bar-in1}
\hat{\Gamma}_{1}^{2}(\slashed{D}_{\hat{M}}^{TN}) \leq \frac{n^2}{4 \operatorname{vol}(\hat{M}^{n})} \int_{\hat{M}^{n}} \hat{H}^{2}dv,
\end{equation}\begin{equation}\label{Bar-in2}
\tilde{\Gamma}_{1}^{2}(\slashed{D}_{\tilde{M}}^{TN})\leq \frac{n^2}{4}+\frac{n^2}{4\operatorname{vol}(\tilde{M}^{n})} \int_{\tilde{M}^{n}} \tilde{H}^2dv,
\end{equation}and \begin{equation}\label{Bar-in3}\bar{\Gamma}_{1}^{2}(\slashed{D}_{\tilde{M}}^{TN})\leq\frac{n^2}{4}\left(1+\inf_{\sigma\in\Pi}\max_{M^{n}}|\bar{H}|\right)^{2},\end{equation}respectively, where $\sigma$ denotes the isometric embedding from $\bar{M}^{n}$ to the hyperbolic space $\mathbb{H}^{n+1}(-1)$. Furthermore, inequality \eqref{Bar-in3} is later improved by Ginoux \cite[Theorem 1]{Gi2} to the following\begin{equation}\label{Gin-in3}\bar{\Gamma}_{1}^{2}(\slashed{D}_{\tilde{M}}^{TN})\leq\frac{n^2}{4}\left(\inf_{\sigma\in\Pi}\max_{M^{n}}|\bar{H}|^{2}-1 \right).\end{equation}
Let $\slashed{D}$ be the Dirac operator of any Dirac bundle $\slashed{S}$ over the compact submanifold $M^{n}\hookrightarrow\mathbb{R}^{n+p}$  with mean curvature $H$ and   $\left\{\Gamma_{j} ; \mathfrak{s}_{j}\right\}_{j \in \mathbb{N}}$ be a spectral resolution of $\slashed{D}^2$. By constructing a suitable test spinors and then applying a theorem of Ashbaugh and Hermi \cite{AH}, Chen  proved the following inequalities \cite[Theorem 3.2]{Che}
\begin{equation}\begin{aligned}\label{ch-gap}
&\sum_{i=1}^k\left(\Gamma_{k+1}-\Gamma_{i}\right)^2\\ &\leq \frac{4}{n} \sum_{i=1}^k\left(\Gamma_{k+1}-\Gamma_{i}\right)\left(\Gamma_{i}+\frac{n^2}{4} \int_MH^2\left\langle\mathfrak{s}_{i}, \mathfrak{s}_{i}\right\rangle dv-\int_M\left\langle\Re \mathfrak{s}_{i}, \mathfrak{s}_{i}\right\rangle dv\right),
\end{aligned}\end{equation}which is an eigenvalue inequality of Yang type and improves \eqref{an-gap}, where $\Re$ is a curvature morphism defined in Subsection \ref{subsec2.1}. In addition, Chen
proved an inequality of low order eigenvalues as follows \cite[Theorem 3.2]{Che}:
\begin{equation}\label{chen-ineq}
\sum_{i=1}^{n}\left(\Gamma_{i+1}-\Gamma_1\right) \leq 4 \Gamma_1+n^2 \int_M\left\langle H^{2} \mathfrak{s}_{i}, \mathfrak{s}_{i}\right\rangle dv-4 \int_M\left\langle\Re \mathfrak{s}_{i}, \mathfrak{s}_{i}\right\rangle dv.
\end{equation}

The remainder of this article is organized as follows. In Section \ref{sec2}, we recall some fundamental facts from submanifold geometry. Subsequently, we introduce several important
concept from spin geometry, such as Dirac operator, Dirac bundle and so on. In addition, we state the famous Bochner-Lichnerowicz-Weitzenb\"{o}ck formula which
will be useful for proving our main result. More precisely, our main result states that,
for any  $j \in \mathbb{N}^{*}$, the eigenvalues of square of the Dirac operator $\slashed{D}^{2}$ satisfy the following inequality
\begin{equation*}
\sum^{n}_{i=1}\Gamma_{i+j} \leq(n+4)\Gamma_{j}+n^{2}\int_{M^{n}}H^{2}\langle\mathfrak{s}_{j}, \mathfrak{s}_{j}\rangle dv-4\int_{M^{n}}\langle\Re \mathfrak{s}_{j}, \mathfrak{s}_{j}\rangle dv.
\end{equation*}

In Section \ref{sec3}, we will prove that on an $n$-dimensional submanifold isometrically immersed into the Euclidean space $\mathbb{R}^{n+p}$ a certain number of Dirac eigenvalues can always be bounded in terms of $\int_{M^{n}} H^{2}dv$ without any a-priori assumption on the spectrum or on the scalar curvature. Our key strategy is, by a rotational coordinate transformation, to give rise to an equivalent base with respect to certain Hilbert space,  and the base is compose of some new test eigensections belonging to the smooth Dirac bundle. Furthermore, the desired estimate is obtained by combining with Bessel inequality. Essentially, our strategy is not directly based on a variation technique just like in \cite{Che} or an abstract algebraic inequality in \cite{LP}. As the previous works, Bochner-Schr\"{o}dinger-Lichnerowicz formula is a crucial tool to get some estimates for eigenvalues of the Dirac operator. In Section \ref{sec4}, we give several interesting applications of our main result.  For example, we establish some eigenvalue inequalities on the compact submanifolds of unit sphere and projective spaces and further get three bounds of general Reilly type. On the meantime, based on the argument for the upper bound, we propose a conjecture on the lower bound of the eigenvalues of Dirac operator and it is closely related to the famous Willmore conjecture. Under certain curvature condition, we also obtain some universal bounds of the Dirac operator. In addition, we also provide an alternative proof of Anghel's result.  In particular, instead of previous method, by establishing an upper bound of Levitin-Parnovski type, we give an upper bound for the sum of the first $n$ nontrivial eigenvalues of Atiyah-Singer Laplacian acting on the spin manifolds without dimensional assumption.

\begin{ack}The research was supported by the National Natural
Science Foundation of China (Grant Nos. 11861036 and 11826213) and the Natural Science Foundation of Jiangxi Province (Grant No. 20224BAB201002).\end{ack}

\section{Preliminaries and Main Results} \label{sec2}

To state the main result, let us fix some terminology in the first three parts in this section.

\subsection{Dirac  Bundle and Dirac Operator}\label{subsec2.1}

The following is some basic facts with respect to the spin geometry in \cite{LM} or in \cite{Fri4}.
Let $M^{n}$ denote an $n$-dimensional compact Riemannian manifold with a smooth Riemannian metric $g$
and $\mathcal{C} \ell(M^{n})$ be the Clifford bundle,  which is equipped with an associative algebra  induced by the tangent bundle $T M^{n}$ and the Riemannian product $g$ on the manifold $M^{n}$ \cite{GL,LM}. It is well known that, there exists a canonical embedding $T M^{n}(\subset \mathcal{C} \ell(M^{n})) \hookrightarrow \mathcal{C} \ell(M^{n})$, and thus both metric and covariant differentiation on $\mathcal{C} \ell(M^{n})$, which is denoted by $\slashed{\nabla}$,  can be extended from $T M^{n}$ to $\mathcal{C} \ell(M^{n})$ such that the covariant differentiation preserves the metric. Under the assumption that tangent vectors act by Clifford multiplication on $\slashed{S}$, i.e. there is a vector bundle homomorphism
\begin{equation*}
T M^{n} \otimes \slashed{S} \rightarrow \slashed{S}, \quad X \otimes \mathfrak{s} \mapsto X\cdot\mathfrak{s},
\end{equation*}
so-called Dirac bundle  $\slashed{S}$ is defined by a bundle of left modules over $\mathcal{C} \ell(M^{n})$ together with the Riemannian metric $g$ and the corresponding connection on $\slashed{S}$ satisfying three properties as follows:

\noindent$\bullet$ the Clifford relations
\begin{equation}\label{cli-equ-1}
X\cdot Y \cdot \mathfrak{s}+Y\cdot X\cdot \mathfrak{s}+2\langle X, Y\rangle \mathfrak{s}=0
\end{equation}
for all $X, Y \in T_p M^{n}, \mathfrak{s} \in \slashed{S}$, where  $p \in M^{n}$.

\noindent $\bullet$ skew symmetry(or say skew adjoint) \begin{equation}\label{cli-equ-2}\left\langle X\cdot  \mathfrak{s}_{1}, \mathfrak{s}_{2}\right\rangle=-\left\langle\mathfrak{s}_{1}, X\cdot  \mathfrak{s}_{2}\right\rangle,\end{equation}

\noindent$\bullet$ the product rule\begin{equation}\label{cli-equ-3}\slashed{\nabla} (\phi \mathfrak{s})=(\slashed{\nabla} \phi) \mathfrak{s}+\phi \slashed{\nabla} \mathfrak{s},\text{ or }\slashed{\nabla}_{X} (\phi \mathfrak{s})=(\slashed{\nabla}_{X} \phi) \mathfrak{s}+\phi \slashed{\nabla}_{X} \mathfrak{s},\end{equation}

\noindent for $\mathfrak{s}_{1}, \mathfrak{s}_{2}, \mathfrak{s} \in C^{\infty}(\slashed{S}), \phi \in \Gamma(\mathcal{C} \ell(M^{n})), X,Y \in C^{\infty}(T M^{n}),$ where $\Gamma(\mathcal{C} \ell(M^{n}))$ stands for the set of all smooth sections on $\mathcal{C} \ell(M^{n})$,
the covariant derivative on $\slashed{S}$ is also denoted by $\slashed{\nabla}$. Any generalized Dirac bundle generates a first-order elliptic differential operator $\slashed{D}: C^{\infty}(\slashed{S}) \rightarrow C^{\infty}(\slashed{S})$, usually called Dirac operator and locally  expressed as
\begin{equation*}
\slashed{D}=\sum_{i=1}^{n} e_{i}\cdot  \slashed{\nabla}_{e_{i}}.
\end{equation*}In what follows, we omit Clifford multiplication ``$\cdot$" for convenience as long as there is no any confusion.
We assume that $C_{0}^{\infty}(M^{n}, \slashed{S})$ is the space of smooth, compactly supported sections of $\slashed{S}$. For any $\mathfrak{s}_{1}, \mathfrak{s}_{2} \in C_{0}^{\infty}(M^{n}, \slashed{S})$, the $L^2$-product and norms are defined,
\begin{equation}\label{inner-product}
(\mathfrak{s}_{1}, \mathfrak{s}_{2})=\int_{M^{n}}\langle\mathfrak{s}_{1}, \mathfrak{s}_{2}\rangle \mathrm{d}v, \quad \text{and}\quad \|\mathfrak{s}_{1}\|_{M^{n}}^2=(\mathfrak{s}_{1}, \mathfrak{s}_{1}),
\end{equation} respectively. The Hilbert space $L^2(M^{n}, \slashed{S})$ is obtained by completing $C_{0}^{\infty}(M^{n}, \slashed{S})$ with respect to the above scalar product $(\cdot,\cdot)$. Pointwise scalar products will be denoted by $\langle\cdot, \cdot\rangle$ and pointwise norms by $|\cdot|$.
For $f \in C^{\infty}(M^{n})$ and smooth section $\mathfrak{s} \in C^{\infty}(\slashed{S}),$ we get
\begin{equation}\label{D2-fs-eq}
\begin{aligned}
\slashed{D}(f \mathfrak{s}) &=\operatorname{grad}(f) \mathfrak{s}+f \slashed{D} \mathfrak{s},\end{aligned}
\end{equation}
\begin{equation}\label{D-2-Labl}
\begin{aligned}\slashed{D}^{2}(f \mathfrak{s}) &=\Delta(f) \mathfrak{s}-2 \slashed{\nabla}_{\operatorname{grad}(f)} \mathfrak{s}+f \slashed{D}^{2} \mathfrak{s},
\end{aligned}
\end{equation}
where $\Delta$ is the positive scalar Laplacian.
For the Dirac operator $\slashed{D},$ one has the following Bochner-Lichnerowicz-Weitzenb\"{o}ck formula \cite{LM,GL}
\begin{equation}\label{BLW-formula}
\slashed{D}^{2}=\slashed{\nabla}^{*} \slashed{\nabla}+ \Re,
\end{equation}
where $\Re$ is the curvature morphism acting on Dirac bundle $\slashed{S}$, which is given by
\begin{equation*}
\begin{aligned}
\Re  =\sum_{i<j} e_{i} e_{j} \mathcal{R} _{e_{i}, e_{j}},\end{aligned}
\end{equation*}and
\begin{equation*}
\begin{aligned}\mathcal{R} _{e_{i}, e_{j}} &=\left[\slashed{\nabla}_{i}, \slashed{\nabla}_{j}\right]-\slashed{\nabla}_{\left[e_{i}, e_{j}\right]}.
\end{aligned}
\end{equation*}
Let $\slashed{E}$ be a subbundle of $\slashed{S}$ invariant under $\slashed{\nabla}$ and $\mathfrak{R}$, i.e., \begin{equation*}\slashed{\nabla}_{X}\left(C^{\infty}(\slashed{E})\right)\subset C^{\infty}(\slashed{E}),\end{equation*} for any vector field $X$, and $\mathfrak{R}(\slashed{E}) \subset \slashed{E}$ respectively. Therefore, by the Bochner-Lichnerowicz-Weitzenb\"{o}ck formula \eqref{BLW-formula}, it is easy to show that $C^{\infty}(\slashed{E})$ is $\slashed{D}^2$-invariant, and we call it a Dirac invariant subbundle. To avoid some technical(but not essential) difficulties, we always consider the Dirac operator acting on some invariant subbundles $\slashed{E}$ of the Dirac bundle $\slashed{S}$.
There exist three basic cases of the Dirac bundles\cite{Che,An,Ba2,Fri4}. However, we are mainly interesting in the first two cases throughout this paper.

$\bullet$ The Clifford bundle $\mathcal{C} \ell(M^{n}).$  Under the isomorphism $\mathcal{C} \ell(M^{n}) \cong \Lambda^*(M)$, one has $D \cong$ $d+d^*$, which is called a Gauss-Bonnet operator. Therefore, $D^2 \cong d^* d+d d^*$ is  classical Hodge Laplacian on forms. For this case, $\mathfrak{R}_{\mid M^{n} \times \mathbb{C}}=0$, and $\mathfrak{R}_{\mid T M^{n}}= \text{Ric}$, which are the ordinary Ricci transformation.

$\bullet$ The spinor bundles. Suppose that $M^{n}$ is spin manifold with the spin structure on its tangent bundle. Let $\slashed{E}$ be any spinor bundle associated to tangent bundle. Then $\slashed{E}$ carries a canonical Riemannian connection. In this case, $\slashed{D}$ is the classical Dirac operator (which is also called Atiyah-Singer operator) and $\Re =\frac{1}{4} S,$ where $S$ is the scalar curvature of $M^{n}$.

$\bullet$  Twisted Clifford bundle. If $M^{n}$ is spin and $\mathcal{C} \ell(M^{n})(M^{n})=\slashed{E} \otimes E$, where $E$ are some Hermitian coefficient bundles, we get a ``twisted" Dirac operator.

\subsection{Extrinsic Formulas} Let $x=\left(x^{1}, \ldots, x^{n+p}\right): M^{n} \rightarrow \mathbb{R}^{n+p}$ be an isometric immersion of $M^{n}$ in $\mathbb{R}^{n+p}$. Nash embedding theorem guarantees the existence of isometric immersion $x$. To keep the subscripts uniform, the following conventions on the ranges of indices are used in this article:
\begin{equation*}
1 \leq A, B, C, D \leq n+p, \quad 1 \leq i, j, k, l \leq n, \quad n+1 \leq \alpha, \beta, \gamma \leq n+p .
\end{equation*}We choose a local orthonormal frame field $\left\{e_A\right\}_{A=1}^{n+p}$ in $\mathbb{R}^{n+p}$ with dual coframe field $\left\{\omega_A\right\}_{A=1}^{n+p}$, such that, $\left\{e_{1}, \cdots, e_{n}\right\}$ is a local orthonormal basis of $M^{n}$ with respect to the induced metric $g$. Hence, we have
\begin{equation*}
d x=\sum_{i} \omega_{i} e_{i}, \quad d e_{i}=\sum_{j} \omega_{i j} e_{j}+\sum_{\alpha} \omega_{i \alpha} e_{\alpha}
\end{equation*}
and
\begin{equation*}
d e_{\alpha}=\sum_{i} \omega_{\alpha i} e_{i}+\sum_{\beta} \omega_{\alpha \beta} e_{\beta}.
\end{equation*}
Restrict these forms to $M^{n}$, then we have
\begin{equation}\label{eq-omega-al=0}
\omega_{\alpha}=0 \text { for } n+1 \leq \alpha \leq n+p.
\end{equation}
According to Cartan's lemma and \eqref{eq-omega-al=0}, one can show that
\begin{equation*}
\omega_{i \alpha}=\sum_{j} h_{i j}^{\alpha} \omega_{j}, \quad h_{i j}^\alpha=h_{j i}^{\alpha}.
\end{equation*}
If $X$ and $Y$ are vector fields tangent to $M^{n}$, the second fundamental form $\mathcal{B}(X, Y)$ of $M^{n}$ in $\mathbb{R}^{n+p}$ is defined by Gauss equation
\begin{equation*}
\bar{\nabla}_X Y=\nabla_X Y+\mathcal{B}(X, Y),
\end{equation*}
where $\nabla$ and $\bar{\nabla}$ are the covariant differentiations on $M^{n}$ and $\mathbb{R}^{n+p}$ respectively. It is not difficult to see that, under the previous local orthonormal frames, the second fundamental form $\mathcal{B}$ is expressed as\begin{equation*}
\mathcal{B}=\sum_{\alpha, i, j} h_{i j}^\alpha \omega_{i} \otimes \omega_{j} \otimes e_{\alpha},
\end{equation*}and its squared norm given by \begin{equation*}|\mathcal{B}|^2=\sum_{\alpha, i, j}\left(h_{i j}^\alpha\right)^{2}.\end{equation*}
Assume that \begin{equation}\label{H-mean-H-equa}\boldsymbol{H} =\frac{1}{n}\sum_{\alpha=n+1}^{n+p} H^{\alpha} e_{\alpha}=\frac{1}{n}\sum_{\alpha=n+1}^{n+p}\left(\sum_{i=1}^{n} h_{i i}^{\alpha}\right) e_{\alpha},\ \ {\rm and} \ \ H=\frac{1}{n} \sqrt{\sum_{\alpha=n+1}^{n+p}\left(\sum_{i=1}^{n} h_{i i}^{\alpha}\right)^{2}}\end{equation} are the mean curvature vector field and the mean curvature of $M^{n}$, respectively. Then, mean curvature vector field  $\boldsymbol{H}$ is given by

\begin{equation}\boldsymbol{H}=\operatorname{tr} \mathcal{B}=\sum_{i=1}^{n} \mathcal{B}\left(e_{i}, e_{i}\right)=-\left(\Delta x^1, \ldots, \Delta x^{n}\right).
\end{equation}
Moreover, the induced structure equations of $M^{n}$ are given by
\begin{equation*}
d \omega_{i}=\sum_{j} \omega_{i j} \wedge \omega_{j}, \quad \omega_{i j}=-\omega_{j i}, \end{equation*}and\begin{equation*}
d \omega_{i j}=\sum_{k} \omega_{i k} \wedge \omega_{k j}-\Omega_{ij},
\end{equation*}where

\begin{equation*}\Omega_{ij}=\frac{1}{2} \sum_{k, l} R_{i j k l} \omega_{k} \wedge \omega_{l},\end{equation*}
denotes the curvature $2$-form associated to the Riemannian metric $g$ on $M^{n}$ and $R_{ijkl}$ denotes components of Riemannian curvature tensor of $M^{n}$.
Since the ambient space is Euclidean space, the Gauss equation implies
\begin{equation}\label{Rie-tensor-comp}
R_{i j k l}=\sum_{\alpha}\left(h_{i k}^\alpha h_{j l}^\alpha-h_{i l}^\alpha h_{j k}^\alpha\right).
\end{equation}
 From \eqref{Rie-tensor-comp}, components of the Ricci curvature of $M^{n}$ are given by
\begin{equation}\label{Ricci-component}
R_{i k}=\sum_{\alpha} H^\alpha h_{i k}^\alpha-\sum_{\alpha, j} h_{i j}^\alpha h_{j k}^\alpha.
\end{equation}Taking trace for \eqref{Ricci-component} leads to

\begin{equation}\label{Gauss-for}
S=n^{2}H^{2}-|\mathcal{B}|^{2},
\end{equation}
where $S$ denotes the scalar curvature of $M^{n}$.
 By a  straight calculation, one concludes that \cite{CC}, pointwise on $M^{n}$,
\begin{equation}
\begin{aligned}
\sum^{n+p}_{A=1}(\Delta x_{A})^{2}=n^{2}H^{2},
\end{aligned}
\end{equation}
\begin{equation}
\begin{aligned}
\sum^{n+p}_{A=1}\Delta x_{A}\nabla  x_{A}= 0,
\end{aligned}
\end{equation}and
\begin{equation}\label{cc-ineq-1}
\sum^{n+p}_{A=1}\langle\operatorname{grad}  (x_{A}),\operatorname{grad}  (x_{A})\rangle= n.
\end{equation}
For any functions  $u, w\in C^{1}(M^{n})$, a direct computation implies that

\begin{equation}
\begin{aligned}
\sum^{n+p}_{A=1}\langle\nabla  x_{A},\nabla  u\rangle\langle\nabla
x_{A},\nabla  w\rangle=\langle\nabla  u,\nabla  w\rangle.
\end{aligned}
\end{equation}

\subsection{Eigenvalue problem}  The restriction of $\slashed{D}^{2}$ to $C^{\infty}(\slashed{E})$ has a discrete spectrum.
It is well known that the Dirac operator $\slashed{D}$ of any Dirac bundle is self-adjoint with respect to the inner product:
 \begin{equation*}\int_{M^{n}}\langle\mathfrak{s}_{1},\mathfrak{s}_{2}\rangle dv,\end{equation*}
 for any $\mathfrak{s}_{i}\in C^{\infty}(\slashed{E})$, where $i=1,2,$
and elliptic on the compact manifold. Therefore, $\slashed{D}^{2}$ has a spectrum contained in the positive part of $\mathbb{R}$ numbered like
\begin{equation*}
0 \leq \Gamma_{1} \leq \Gamma_{2} \leq \cdots \nearrow \infty
\end{equation*}
and one can find an orthonormal basis $\left\{\mathfrak{s}_{j}\right\}_{j \in \mathbb{N} }$ of $L^{2}(\slashed{E})$ consisting of eigenfunctions of $\slashed{D}^{2}$ (i.e. $\left.\slashed{D}^{2} \mathfrak{s}_{j}=\Gamma_{j} \mathfrak{s}_{j}, j \in \mathbb{N} \right) .$ Such a system $\left\{\Gamma_{j} ; \mathfrak{s}_{j}\right\}_{j \in \mathbb{N} }$ is called a spectral decomposition of $L^{2}(\slashed{E})$ generated by $\slashed{D}^{2}$, or, in short, a spectral resolution of $\slashed{D}^{2}$. For convenience, we use $\bar{\Gamma}_{i}$ to express the $i^{\rm th}$ nonzero eigenvalue of Dirac operator $\slashed{D}^{2}$ , and thus have
\begin{equation*}
0 < \bar{\Gamma}_{1} \leq \bar{\Gamma}_{2} \leq \cdots \nearrow \infty.
\end{equation*}
\subsection{Main Results}

Let $\Omega$ is a bounded domain on an $n$-dimensional Riemannian manifold $M^{n}$ with piecewise smooth boundary $\partial\Omega$. The famous fixed membrane problem states as follows:
\begin{equation}\label{diri-prob} {\begin{cases} \
 \Delta u +\lambda u=0, \ \ & {\rm in} \ \ \ \ \Omega, \\
  \ u=0, \ \ & {\rm on} \ \ \partial \Omega,
\end{cases}}\end{equation}where $\Delta$ denotes the Laplacian on $M^{n}$. We use $\lambda_{i}$ to denote the $i$-th eigenvalue of Laplacian on $M^{n}$.
Based on an extrinsic method, Chen and Cheng \cite{CC} obtained an extrinsic bound. Specially, let $\psi:M^{n}\rightarrow  \mathbb{R}^{n+p}$ be an immersed submanifold with mean curvature $H$ and $\Pi$ denote the set of all the isometric immersions from $M^{n}$ into a Euclidean space, they proved
\begin{equation}
\begin{aligned}
\label{c-cheng-1} \sum^{k}_{i=1}(\mu_{k+1}-\mu_{i})^{2}
\leq\frac4n\sum^{k}_{i=1}(\mu_{k+1}-\mu_{i})\mu_{i},
\end{aligned}
\end{equation}
where \begin{equation}\label{mu-i}\mu_{i}=\lambda_{i}+\frac{1}{4}\inf_{\psi\in\Pi}\max_{\Omega}n^{2}H^{2},\end{equation}
which generalizes a very sharp eigenvalue inequality due to Yang \cite{Y} (cf. \cite{CY2}):
\begin{equation}\label{y1-ineq}\sum^{k}_{i=1}(\lambda_{k+1}-\lambda_{i})^{2}\leq\frac{4}{n}\sum^{k}_{i=1}(\lambda_{k+1}-\lambda_{i})\lambda_{i},\end{equation}
 for a bounded domain on
$\Omega\subset\mathbb{R}^{n}$. A cerebrated conjecture states as follows (see \cite{Ash3}):
\begin{GPPW}Let $\Omega$ be a bounded domain on an $n$-dimensional Euclidean space $\mathbb{R}^{n}$. Assume that $\lambda_{i}$ is the $i$-th eigenvalue of Dirichlet problem
\eqref{diri-prob} for the Laplacian on $\mathbb{R}^{n}$. Then, the following inequality
\begin{equation}
 \lambda_{2} +\lambda_{3} +\cdots+ \lambda_{n+1}
\leq n[\lambda_{2}(\mathbb{B}^{n})/\lambda_{1}(\mathbb{B}^{n})]\lambda_{1}\end{equation}holds, where $\lambda_{i}(\mathbb{B}^{n})(i=1,2)$ denotes the $i^{th}$ eigenvalue of Laplacian on the ball $\mathbb{B}^{n}$ with the same volume as the bounded domain $\Omega$, i.e., $\operatorname{vol}(\Omega)=\operatorname{vol}(\Omega^{\ast})$.\end{GPPW} Attacking this conjecture,  Ashbaugh and Benguria
\cite{AB1} made a fundamental contribution for establishing a surprising universal inequality as follows:
\begin{equation}\label{1.16}
\lambda_{2} +\lambda_{3} +\cdots+ \lambda_{n+1}\leq (n + 4)\lambda_{1}.\end{equation} for $\Omega\subset\mathbb{R}^{n}$, in 1993. Up to now, this conjecture is so far unresolved. For more references on the solution of
this conjecture, we refer the readers to
\cite{AB1,AB2,AB3,CZh,HP,PPW} and references therein. In particular, an amazing breakthrough was made by Ashbaugh and Benguria in \cite{AB2}(or see\cite{AB3}). By dealing with some good properties of Bessel
functions, they  affirmatively settled the general Payne, P\'{o}lya and Weinberger's Conjecture under certain special case. In 2002, by establishing  an algebraic inequality, Levitin and Parnovski \cite{LP}  generalized  \eqref{1.16}  to the following:
\begin{equation}\label{LP1}
 \lambda_{j+1} +\lambda_{j+2} +\cdots+ \lambda_{j+n}\leq (n + 4)\lambda_{j},
\end{equation}
where $j$ is any positive integer. For general Riemannian manifold $M^{n}$ isometrically immersed into the $(n+p)$-dimensional Euclidean space $\mathbb{R}^{n+p}$,  Chen and Cheng   \cite{CC} obtained
\begin{equation}
\begin{aligned}
\label{c-cheng-2} \mu_{2} +\mu_{3} +\cdots+\mu_{n+1}
\leq (n + 4)\mu_{1},
\end{aligned}
\end{equation} where $\mu_{i}$ is given by \eqref{mu-i}.
Motivated by previous works, we would like to prove the following theorem in this paper and its proof will be left to the next section.

\begin{thm}\label{thm-main} Let $M^{n}$ be a compact Riemannian manifold with dimension $n$ and $F: M^{n} \longrightarrow \mathbb{R}^{n+p}$ be an isometric embedding with mean curvature $H$. Let $\slashed{D}$ be the Dirac operator of any Dirac invariant subbundle $\slashed{E}$ over $M^{n}$ and let $\left\{\Gamma_{j} ; \mathfrak{s}_{j}\right\}_{j \in \mathbb{N}}$ be a spectral resolution of $\slashed{D}^{2}=\slashed{\nabla}^{*} \slashed{\nabla}+ \Re: C^{\infty}(\slashed{E})\rightarrow C^{\infty}(\slashed{E})$. Then we have, for any  $j \in \mathbb{Z}^{+}$,
\begin{equation}\label{eq-main}
\sum^{n}_{k=1}\Gamma_{j+k} \leq(n+4)\Gamma_{j}+n^{2}\int_{M^{n}}H^{2}\langle\mathfrak{s}_{j}, \mathfrak{s}_{j}\rangle dv-4\int_{M^{n}}\langle\Re \mathfrak{s}_{j}, \mathfrak{s}_{j}\rangle dv.
\end{equation}
\end{thm}

\begin{rem}When $j=1$, inequality \eqref{eq-main} deduces to inequality \eqref{chen-ineq}. Thus, Theorem {\rm \ref{thm-main}} can be viewed as an extension of
 Chen's eigenvalue inequality established in {\rm \cite{Che}} and this generalization leads to some wider applications.  \end{rem}

\begin{rem}
The Nash's theorem guarantees the existence of isometrical embedding from complete Riemannian manifold to an Euclidean space and the existence will
play  a very critical role in finding an appropriate test spinors to establish an extrinsic bounds for the eigenvalues. According to a similar argument as in {\rm \cite{Col}},
we know that there exist no universal
inequalities for the eigenvalues of the Dirac operator
in the absence of any other conditions of intrinsic or extrinsic
geometry. Therefore, the spectrum of the Dirac operators naturally contains information
about the extrinsic geometry of a submanifold when it is embedded into certain
Euclidean spaces.  In other words, the extrinsic character of the estimates is reflected in their dependence on the geometry of the immersion chosen. \end{rem}

\section{Proof of Theorem \ref{thm-main}}\label{sec3}
In this section, we shall give the proof of Theorem \ref{thm-main}. To this end, we need some key lemmas.
\begin{prop}
\label{prop2.1} Let $\Psi$ be a  smooth function  on an $n$-dimensional
compact  Riemannian manifold $M^{n}$. Let $\slashed{D}$ be the Dirac operator of any Dirac bundle $\slashed{E}$ over $M^{n}$ and let $\left\{\Gamma_{j} ; \mathfrak{s}_{j}\right\}_{j \in \mathbb{N}}$ be a spectral resolution of $\slashed{D}^{2}$. Then, for any $j=1, 2, \cdots,$  we have
\begin{equation}\label{3.1}
\int_{M^{n}}\left(\slashed{D}^{2} \Psi  \mathfrak{s}_{j}-2 \slashed{\nabla}_{\operatorname{grad}\left(\Psi \right)} \mathfrak{s}_{j} \right) \Theta_{l}dv
=\sum^{\infty}_{k=1}(\Gamma_{k}-\Gamma_{j})\alpha_{jk}^{2},
\end{equation}
where $\mathfrak{s}_{k}$ is an orthonormal eigenfunction  corresponding to
eigenvalue $\Gamma_{k}$ and
\begin{equation}\label{alpha-jk}
\alpha_{jk}=\int_{M^{n}}\Psi \mathfrak{s}_{j}\mathfrak{s}_{k}dv.
\end{equation}
\end{prop}

\begin{proof}
Since $\left\{\Gamma_{j} ; \mathfrak{s}_{j}\right\}_{j \in \mathbb{N}}$ is a spectral resolution of $\slashed{D}^{2}$, $\{\mathfrak{s}_{k}\}^{\infty}_{k=1}$ forms an orthonormal basis of the Hilbert space $L^2(M^{n}, \slashed{E})$ associated with inner product $(,)$ given by \eqref{inner-product}, which means that, for any positive integer $j$, the linear representation of $\Psi \mathfrak{s}_{j}$ in term of this base $\{\mathfrak{s}_{k}\}^{\infty}_{k=1}$ is given by
\begin{equation}\label{2.2}
\Psi \mathfrak{s}_{j}=\sum^{\infty}_{k=1}\alpha_{jk}\mathfrak{s}_{k}.
\end{equation}
Noticing the completeness of the base $\{\mathfrak{s}_{k}\}^{\infty}_{k=1}$, we can applying classical Parseval's identity to conclude that
\begin{equation}\label{norm-Psi}
\|\Psi \mathfrak{s}_{j}\|^{2}_{M^{n}}=\sum^{\infty}_{k=1}\alpha_{jk}^{2}.
\end{equation} By taking $f=\Psi$ and $\mathfrak{s}=\mathfrak{s}_{j}$ in \eqref{D-2-Labl}, we know that
\begin{equation}\label{D2-eq2}
\slashed{D}^{2}(\Psi \mathfrak{s}_{j})=\mathfrak{s}_{j}\Delta \Psi-2 \slashed{\nabla}_{\operatorname{grad}\Psi} \mathfrak{s}_{j}+\Psi   \slashed{D}^{2} \mathfrak{s}_{j}.
 \end{equation}
A straightforward calculation shows that
\begin{equation}\label{int-M-eq1}
\begin{aligned}
&\quad~\int_{M^{n}}\left[\slashed{D}^{2}(\Psi \mathfrak{s}_{j})-\Psi\slashed{D}^{2}\mathfrak{s}_{j}\right]\mathfrak{s}_{k}dv
\\&=\int_{M^{n}}\left[\mathfrak{s}_{k}\slashed{D}^{2}(\Psi \mathfrak{s}_{j})-\Psi \mathfrak{s}_{k}\slashed{D}^{2}\mathfrak{s}_{j}\right]dv
\\&=\int_{M^{n}}\left(\Gamma_{k}\Psi \mathfrak{s}_{j}\mathfrak{s}_{k}-\Gamma_{j}\Psi \mathfrak{s}_{j}\mathfrak{s}_{k}\right)dv
\\&=-(\Gamma_{j}-\Gamma_{k})\int_{M^{n}}\Psi \mathfrak{s}_{j}\mathfrak{s}_{k}dv.
\end{aligned}
\end{equation}
Therefore, combining \eqref{int-M-eq1} with \eqref{D2-eq2} yields
\begin{equation}
\label{2.3}\int_{M^{n}}\left(\mathfrak{s}_{j}\Delta \Psi -2\slashed{\nabla}_{\operatorname{grad}
\Psi}\mathfrak{s}_{j}\right)\mathfrak{s}_{k}dv
=(\Gamma_{k}-\Gamma_{j})\int_{M^{n}}\Psi \mathfrak{s}_{j}\mathfrak{s}_{k}dv.
\end{equation}
Furthermore, we deduce from   \eqref{alpha-jk} and   \eqref{norm-Psi} that,
\begin{equation}
\begin{aligned}
\label{2.4}
&\sum^{\infty}_{k=1}(\Gamma_{k}-\Gamma_{j})\alpha_{jk}^{2}
\\&=\sum^{\infty}_{k=1}(\Gamma_{k}-\Gamma_{j})\left(\int_{M^{n}}\Psi \mathfrak{s}_{j}\mathfrak{s}_{k}dv\right)^{2}
\\&=\sum^{\infty}_{k=1}\Gamma_{k}\left(\int_{M^{n}}\Psi \mathfrak{s}_{j}\mathfrak{s}_{k}dv\right)^{2}-\Gamma_{j}\|\Psi \mathfrak{s}_{j}\|_{M^{n}}^{2}.
\end{aligned}\end{equation}
Utilizing (\ref{2.2}) and noticing
\begin{equation*}
\slashed{D}^{2}(\Psi \mathfrak{s}_{j})
=\sum^{\infty}_{k=1}\alpha_{jk}\slashed{D}^{2}\mathfrak{s}_{k}=\sum^{\infty}_{k=1}\alpha_{jk}\Gamma_{k}\mathfrak{s}_{k},
\end{equation*}
we yield
\begin{equation}
\label{2.5} \Psi \mathfrak{s}_{j}\slashed{D}^{2}(\Psi \mathfrak{s}_{j})
=\sum^{\infty}_{k=1}\alpha_{jk}\Gamma_{k}\mathfrak{s}_{k}\Psi \mathfrak{s}_{j}.
\end{equation}
It follows from \eqref{2.5} that
\begin{equation}
\begin{aligned}
\label{2.6} &\quad~\int_{M^{n}}\Psi \mathfrak{s}_{j}\slashed{D}^{2}(\Psi \mathfrak{s}_{j})dv
\\&=\sum^{\infty}_{k=1}\int_{M^{n}}\alpha_{jk}\Gamma_{k}\mathfrak{s}_{k}\Psi \mathfrak{s}_{j}dv
\\&=\sum^{\infty}_{k=1}\Gamma_{k}\int_{M^{n}}\Psi \mathfrak{s}_{j}\mathfrak{s}_{k}dv\int_{M^{n}}\mathfrak{s}_{k}\Psi \mathfrak{s}_{j}dv
\\&=\sum^{\infty}_{k=1}\Gamma_{k}\left(\int_{M^{n}}\Psi \mathfrak{s}_{j}\mathfrak{s}_{k}dv\right)^{2}.
\end{aligned}
\end{equation}
Hence, from  (\ref{2.4}) and (\ref{2.6}), we have
\begin{equation*}
\begin{aligned}
\sum^{\infty}_{k=1}(\Gamma_{k}-\Gamma_{j})\alpha_{jk}^{2}
&=\int_{M^{n}}(\Psi \mathfrak{s}_{j}\slashed{D}^{2}(\Psi \mathfrak{s}_{j})-\Gamma_{j}\Psi^{2}\mathfrak{s}_{j}^{2})dv
\\&=\int_{M^{n}}(\slashed{D}^{2}(\Psi \mathfrak{s}_{j})-\Gamma_{j}\Psi\mathfrak{s}_{j})\Psi \mathfrak{s}_{j}dv
\\&=\int_{M}\left(\Delta \Psi  \mathfrak{s}_{j}-2 \slashed{\nabla}_{\operatorname{grad}\Psi} \mathfrak{s}_{j} \right) \Theta_{l}dv,
\end{aligned}
\end{equation*}which is the result we desire.
This proves the assertion.
\end{proof}

\begin{prop}

\label{prop2} Let $\Psi_{l}$, $l=1,2,\cdots,m$, be $m$ smooth functions on an $n$-dimensional
compact  Riemannian manifold $M^{n}$. Let $\slashed{D}$ be the Dirac operator of any Dirac bundle $\slashed{E}$ over $M^{n}$ and let $\left\{\Gamma_{j} ; \mathfrak{s}_{j}\right\}_{j \in \mathbb{N}}$ be a spectral resolution of $\slashed{D}^{2}$. Then, for any $j=1, 2, \cdots$,  there exists an
orthogonal matrix $P=(p_{lt})_{m\times m}$ such that $\Phi_{l}=\sum_{t=1}^mp_{lt}\Psi_t$
satisfy
\begin{equation}\begin{aligned}\label{2.7}
\sum^{m}_{l=1}(\Gamma_{j+l}-\Gamma_{j})&\int_{M^{n}}\left(\mathfrak{s}_{j}\Delta \Phi_{l}-2 \slashed{\nabla}_{\operatorname{grad}\Phi_{l}} \mathfrak{s}_{j} \right) \Theta_{l}dv \\&\leq
\sum^{m}_{l=1}\left\|\mathfrak{s}_{j}\Delta \Phi_{l} -2\slashed{\nabla}_{\operatorname{grad}
\Phi_{l}}\mathfrak{s}_{j}\right\|^{2}_{M^{n}},
\end{aligned}\end{equation}
where $\mathfrak{s}_{j}$ is an orthonormal eigenfunction corresponding to
eigenvalue $\Gamma_{j}$, $\Theta_{l}=\Phi_{l} \mathfrak{s}_{j}$
and $l=1,2,\cdots,m$.\end{prop}

\begin{proof} For any $j=1, 2, \cdots$,
we consider an $m\times m$-matrix $A$ defined by:
\begin{equation*}
A:=\left(\int_{M^{n}}\left(\mathfrak{s}_{j}\Delta \Psi_{l} -2\slashed{\nabla}_{\operatorname{grad}
\Psi_{l}}\mathfrak{s}_{j}\right)\mathfrak{s}_{j+t}dv\right)_{m\times m}.
\end{equation*}
From Gram-Schmidt orthogonalization, we can choose an orthogonal
matrix $P=(p_{lt})$ such that
\begin{equation*}
Q=PA=(q_{lt})_{m\times m}=\left(\begin{array}{cccc}
q_{11} & q_{12} & \cdots & q_{1 m} \\
0 & q_{22} & \cdots & q_{2 m} \\
\vdots & \vdots & \ddots & \vdots \\
0 & 0 & \cdots & q_{m m}
\end{array}\right)
\end{equation*}
is an upper triangular matrix, that is,
\begin{equation*}
\begin{aligned}
q_{lt}&=\sum_{i=1}^{m}p_{li}\int_{M^{n}}\left(\mathfrak{s}_{j}\Delta \Psi_{i} -2\slashed{\nabla}_{\operatorname{grad}
\Psi_{i}}\mathfrak{s}_{j}\right)\mathfrak{s}_{j+t}dv\\
&= \int_{M^{n}}\left(\mathfrak{s}_{j}\Delta \left(\sum_{i=1}^{m}p_{li}\Psi_{i}\right)
-2\slashed{\nabla}_{\operatorname{grad}\left(\sum_{i=1}^{m}p_{li}\Psi_{i}\right)}\mathfrak{s}_{j} \right)\mathfrak{s}_{j+t}dv,
\end{aligned}
\end{equation*}
with $Q_{lt}=0$ for $l>t$.  It is evident that, for
\begin{equation*}
\Phi_{l}=\sum_{i=1}^{m}p_{li}\Psi_{i},
\end{equation*}
we have
\begin{equation*}
\begin{aligned}
T_{lt}&=\displaystyle\int_{M^{n}}\left(\mathfrak{s}_{j}\Delta \Phi_{l} -2\slashed{\nabla}_{\operatorname{grad}
\Phi_{l}}\mathfrak{s}_{j}\right)\mathfrak{s}_{j+t}dv=0, \ \text{for
$l>t$}.
\end{aligned}
\end{equation*}
Applying the proposition \ref{prop2.1} to functions $\Phi_{l}$,
one can infer that
\begin{equation}
\begin{aligned}\label{2.8}
&\int_{M^{n}}\left(\Delta \Phi_{l} \mathfrak{s}_{j}-2 \slashed{\nabla}_{\operatorname{grad}\Phi_{l}} \mathfrak{s}_{j} \right) \Theta_{l}dv \\=&\sum^{\infty}_{k=1}(\Gamma_{k}-\Gamma_{j})\mu^{2}_{ljk}\\
=&\sum^{j-1}_{k=1}(\Gamma_{k}-\Gamma_{j})\mu^{2}_{ljk}
+\sum^{j+l-1}_{k=j}(\Gamma_{k}-\Gamma_{j})\mu^{2}_{ljk}
+\sum^{\infty}_{k=j+l}(\Gamma_{k}-\Gamma_{j})\mu^{2}_{ljk},
\end{aligned}
\end{equation}where $\mu_{ljk}$ is defined by \begin{equation*}
\mu_{ljk}:=\int_{M^{n}}\Phi_{l}\mathfrak{s}_{j}\mathfrak{s}_{k}dv.
\end{equation*}
According to \eqref{2.3},  in place of $\Psi$ with  $\Phi_{l}$, we have
\begin{equation}\label{2.9}
\int_{M^{n}}\left(\mathfrak{s}_{j}\Delta \Phi_{l} -2\slashed{\nabla}_{\operatorname{grad}
\Phi_{l}}\mathfrak{s}_{j}\right)\mathfrak{s}_{k}dv
=(\Gamma_{j}-\Gamma_{k})\int_{M^{n}}\Phi_{l}\mathfrak{s}_{j}\mathfrak{s}_{k}dv.
\end{equation}
As a consequence, we yield
\begin{equation*}
(\Gamma_{k}-\Gamma_{j})\int_{M^{n}}\Phi_{l}\mathfrak{s}_{j}\mathfrak{s}_{k}dv=(\Gamma_{k}-\Gamma_{j})\mu_{ljk}=0, \
\text{for $k=j, j+1, \cdots, j+l-1$}.
\end{equation*}
From  (\ref{2.8}), we conclude
\begin{equation}
\begin{aligned}\label{2.10}
\int_{M^{n}}\left(\Delta \Phi_{l} \mathfrak{s}_{j}-2 \slashed{\nabla}_{\operatorname{grad}\Phi_{l}} \mathfrak{s}_{j} \right) \Theta_{l}dv &\leq
\sum^{\infty}_{k=j+l}(\Gamma_{k}-\Gamma_{j})\mu^{2}_{ljk}.
\end{aligned}
\end{equation}
Moreover, we obtain, from (\ref{2.9}), (\ref{2.10}) and Parseval's
identity,
\begin{equation*}
\begin{aligned}
&\sum^{m}_{l=1}(\Gamma_{j+l}-\Gamma_{j})\int_{M^{n}}\left(\Delta \Phi_{l} \mathfrak{s}_{j}-2 \slashed{\nabla}_{\operatorname{grad}\Phi_{l}} \mathfrak{s}_{j} \right) \Theta_{l}dv\\&
\leq \sum^{m}_{l=1}\sum^{\infty}_{k=j+l}(\Gamma_{k}-\Gamma_{j})^2\mu^{2}_{ljk}\\
&\leq \sum^{m}_{l=1} \left\|\mathfrak{s}_{j}\Delta \Phi_{l} -2\slashed{\nabla}_{\operatorname{grad}
\Phi_{l}}\mathfrak{s}_{j}\right\|^{2}_{M^{n}}.
\end{aligned}
\end{equation*}
Therefore, we prove the assertion.
\end{proof}

In order to prove Theorem \ref{thm-main}, we need the famous Nash's embedding
Theorem.
\begin{Nash} Each complete Riemannian manifold $M^{n}$ can be isometrically immersed into a Euclidian space $\mathbb{R}^{n+p}$.\end{Nash}

\noindent
{\it Proof of Theorem {\rm \ref{thm-main}}}.   Nash's
Theorem implies that there exists an isometric immersion from
$M^{n}$ into $\mathbb{R}^{n+p}$. Let $x_{1}, \cdots, x_{n+p}$ be
coordinate functions of $\mathbb{R}^{n+p}$. Then  $x_{1}, \cdots,
x_{n+p}$ are defined on $M^{n}$ globally. It is easy to check that
\begin{equation*}
\Phi_{B}=\sum_{A=1}^{n+p}p_{B A}x_{A},
\end{equation*}
satisfies Proposition
\ref{prop2} because $P=(p_{AB})$ is an
orthogonal matrix of $(n+p)\times(n+p)$-type.
 Applying Proposition
\ref{prop2} to functions $\Psi_{A}=x_{A}$, we obtain
\begin{equation}
\begin{aligned}\label{4.1}
&\sum^{n+p}_{A=1}(\Gamma_{j+A}-\Gamma_{j})\int_{M^{n}}\left(\Delta \Phi_{A} \mathfrak{s}_{j}-2 \slashed{\nabla}_{\operatorname{grad}\left(\Phi_{A}\right)} \mathfrak{s}_{j} \right) \Theta_{A}dv \\&\leq
\sum^{n+p}_{A=1}\left\|\mathfrak{s}_{j}\Delta \Phi_{A} -2\slashed{\nabla}_{\operatorname{grad}
\Phi_{A}}\mathfrak{s}_{j}\right\|^{2}_{M^{n}},
\end{aligned}
\end{equation}
with
\begin{equation*}
\Phi_{A}=\sum_{B=1}^{n+p}p_{AB}x_{B}.
\end{equation*}
  A direct calculation shows that
\begin{equation}\label{3-eq-16-D-int}
\begin{aligned}
&\quad \int_{M^{n}}\left((\Delta \Phi_{A}) \mathfrak{s}_{j}-2 \slashed{\nabla}_{\operatorname{grad}\left(\Phi_{A}\right)} \mathfrak{s}_{j} \right) \Theta_{A}dv\\  &=
\int_{M^{n}}\left((\Delta x_{A}) \mathfrak{s}_{j}-2 \slashed{\nabla}_{\operatorname{grad}\left(x_{A}\right)} \mathfrak{s}_{j} \right) x_{A}\mathfrak{s}_{j}dv\\&
=\left\| \operatorname{grad}\left(x_{A}\right)  \mathfrak{s}_{j} \right\|^{2}_{M^{n}},
\end{aligned}
\end{equation}since $\Theta_{A}=x_{A}\mathfrak{s}_{j}$.
By an orthogonal transformation, it is not hard
to prove, for any $A=1,2,\cdots,n+p$,
\begin{equation}\label{na-Phi}
|\slashed{\nabla} \Phi_{A}|^{2}\leq 1.
\end{equation}
From \eqref{cc-ineq-1} and \eqref{na-Phi}, we have
\begin{equation}
\begin{aligned}
&\quad~\sum^{n+p}_{A=1}(\Gamma_{j+A}-\Gamma_{j})\left\| \left(\operatorname{grad} x_{A}\right)  \mathfrak{s}_{j} \right\|^{2}_{M^{n}}\\
&=\sum^{n}_{i=1}(\Gamma_{i+j}-\Gamma_{j})\left\| \left(\operatorname{grad} x_{i}\right)  \mathfrak{s}_{j} \right\|^{2}_{M^{n}}
+\sum^{n+p}_{\alpha=n+1}(\Gamma_{j+\alpha}-\Gamma_{j})\left\| \left(\operatorname{grad} x_{A}\right)  \mathfrak{s}_{j} \right\|^{2}_{M^{n}}\\
&\geq\sum^{n}_{i=1}(\Gamma_{i+j}-\Gamma_{j})\left\| \left(\operatorname{grad} x_{i}\right)  \mathfrak{s}_{j} \right\|^{2}_{M^{n}}
+(\Gamma_{j+n+1}-\Gamma_{j})\sum^{n+p}_{\alpha=n+1}\left\| \left(\operatorname{grad} x_{\alpha}\right)  \mathfrak{s}_{j} \right\|^{2}_{M^{n}}\\
&=\sum^{n}_{i=1}(\Gamma_{i+j}-\Gamma_{j})\left\| \left(\operatorname{grad} x_{i}\right)  \mathfrak{s}_{j} \right\|^{2}_{M^{n}}
\\&\ \ \ \ \ \ +(\Gamma_{j+n+1}-\Gamma_{j})\int_{M^{n}}\left(\sum_{i=1}^{n}(1-|\left(\operatorname{grad} x_{i}\right)|^{2})\right)\mathfrak{s}_{j}^{2}dv\\
&\geq\sum^{n}_{i=1}(\Gamma_{i+j}-\Gamma_{j})\left\| \left(\operatorname{grad} x_{i}\right)  \mathfrak{s}_{j} \right\|^{2}_{M^{n}}
\\&\ \ \ \ \ \ +\sum^{n}_{i=1}(\Gamma_{i+j}-\Gamma_{j})\int_{M^{n}}\left(1-|\left(\operatorname{grad} x_{i}\right)|^{2}\right)\mathfrak{s}_{j}^{2}dv\\
&=\sum_{i=1}^{n}(\Gamma_{i+j}-\Gamma_{j}).
\end{aligned}
\end{equation}
Consequently, from \eqref{3-eq-16-D-int} and the above inequality, we obtain
\begin{equation}\label{eq-sum-1}
\begin{aligned}
\sum_{i=1}^{n}(\Gamma_{j+i}-\Gamma_{j}) \leq
\sum^{n+p}_{A=1}\left\|\mathfrak{s}_{j}\Delta \Phi_{A} -2\slashed{\nabla}_{\operatorname{grad}
\Phi_{A}}\mathfrak{s}_{j}\right\|^{2}_{M^{n}}.
\end{aligned}
\end{equation}Recall that, applying \eqref{BLW-formula},  Anghel \cite[Lemma 2.9]{An} proved the following lemma.

\begin{lem}\label{lem-An-2}We have
\begin{equation*}
\sum_{A=1}^{n+p}\left\|\Delta x_{A} \mathfrak{s}_{j}-2 \slashed{\nabla}_{\operatorname{grad}x_{A}} \mathfrak{s}_{j}\right\|^{2}_{M^{n}}=4 \Gamma_{j}+ n^{2}\left(H^{2} \mathfrak{s}_{j}, \mathfrak{s}_{j}\right)-4\left(\Re \mathfrak{s}_{j}, \mathfrak{s}_{j}\right).
\end{equation*}  \end{lem}
\noindent The above lemma leads to the following equation:
\begin{equation}\begin{aligned}\label{eq-sum-2}
\sum^{n+p}_{A=1}& \left\|\mathfrak{s}_{j}\Delta \Phi_{A} -2\slashed{\nabla}_{\operatorname{grad}
\Phi_{A}}\mathfrak{s}_{j}\right\|^{2}_{M^{n}}\\&
=4 \Gamma_{j}+n^{2}\int_{M^{n}}H^{2}\langle\mathfrak{s}_{j}, \mathfrak{s}_{j}\rangle dv-4\int_{M^{n}}\langle\Re \mathfrak{s}_{j}, \mathfrak{s}_{j}\rangle dv.\end{aligned}
\end{equation}
Furthermore, substituting \eqref{eq-sum-2} into \eqref{eq-sum-1} yields
\begin{equation*}
\begin{aligned}
\sum_{i=1}^{n}(\Gamma_{j+i}-\Gamma_{j})&\leq
4 \Gamma_{j}+n^{2}\int_{M^{n}}H^{2}\langle\mathfrak{s}_{j}, \mathfrak{s}_{j}\rangle dv-4\int_{M^{n}}\langle\Re \mathfrak{s}_{j}, \mathfrak{s}_{j}\rangle dv,
\end{aligned}
\end{equation*}which gets desired inequality.
 Therefore, we complete the proof of the Theorem \ref{thm-main}.
\begin{equation*}
\eqno\square
\end{equation*}
A straightforward result of Theorem \ref{thm-main} is the following Corollaries.
\begin{corr}\label{thm-main-2} Under the same assumption as Theorem {\rm \ref{thm-main}},
if $\left(\Re \mathfrak{s}_{i}, \mathfrak{s}_{i}\right) \geq \kappa$ for some $\kappa \in \mathbb{R}$, denoting \begin{equation*}\eta_{i}=\Lambda_{i}+\frac{1}{4} \left( \inf_{F\in\Pi}\sup _{M^{n}}n^{2}H^{2}-4\kappa\right),\end{equation*} then, for any $j=1,2,\cdots$, we have
\begin{equation*}
\begin{gathered}
\sum_{i=1}^{n}\left(\eta_{i+j}-\eta_{j}\right) \leq 4 \eta_{j} .
\end{gathered}\end{equation*}\end{corr}
\begin{proof}By Theorem {\rm \ref{thm-main}}, the conclusion is obvious.\end{proof}
\begin{corr}\label{Eulidean-coro-2} The hypotheses being the same as in Theorem {\rm \ref{thm-main}}, if $\mathfrak{R}|_{\slashed{E}} \geq c_{2}Id$, for some $c_{2} \in \mathbb{R}$, then,  for any $j=1,2,\cdots$,  one has
\begin{equation*}
\sum^{n}_{i=1}\Lambda_{i+j} \leq(n+4)\Lambda_{j}+c_{1}-4c_{2},
\end{equation*}
where \begin{equation*}c_{1}=\inf_{F\in\Pi}\max_{M^{n}}n^{2}H^{2}.\end{equation*}\end{corr}

\begin{proof} Since $\mathfrak{s}_{l}$ are unit vectors in $L^2(\slashed{E})$ and \begin{equation*}\langle\mathfrak{R} \mathfrak{s}, \mathfrak{s}\rangle \geq c_{2}\langle \mathfrak{s}, \mathfrak{s}\rangle,\end{equation*} the assertion follows immediately from Theorem {\rm \ref{thm-main}}. The best constant $c_{2}$ is the infimum over $M^{n}$ of the minimum eigenvalue function of $\mathfrak{R}|_{\slashed{E}}$.\end{proof}
\begin{rem}In particular, under the same assumption as Corollary {\rm \ref{Eulidean-coro-2}}, if the immersion is minimal from $M^{n}$ to the unit sphere $\mathbb{S}^{n+1}(1)\subset \mathbb{R}^{n+2}$, we shall get a universal bound as follows:
\begin{equation*}
\sum^{n}_{i=1}\Lambda_{i+j} \leq(n+4)\Lambda_{j}+n^{2}-4c_{2}.
\end{equation*}\end{rem}
\section{Some Applications}\label{sec4}
In this section, we would like to give several applications of Theorem \ref{thm-main}.
\subsection{Submanifolds on Euclidean Spaces}

Assume now that $\Re|_{\slashed{E}}=0$, and let $\mathfrak{s}$ be an eigensection of $\slashed{D}^{2}$ with eigenvalue $0$. Since $\slashed{\nabla}$ preserves the metric, $\langle \mathfrak{s}, \mathfrak{s}\rangle$ is a constant function on $M^{n}$.

\begin{corr}[\textbf{General Reilly type formula I}]\label{Rie-ty-for-1} In addition to the hypotheses of Theorem {\rm \ref{thm-main}}, we further assume that  curvature morphism $\Re|_{\slashed{E}}=0$, and zero is in the spectrum of $\slashed{D}^{2}$ on $C^{\infty}(\slashed{E})$, and the dimension ${\rm dim}(\slashed{E}_{0})=m$, where
$\slashed{E}_{0}$ stands for the zero-eigenspace of $\slashed{D}^{2}$ on $C^{\infty}(\slashed{E})$. Then the first nonzero $n$ eigenvalues of $\slashed{D}^{2}$ on $C^{\infty}(\slashed{E})$ satisfy
\begin{equation}\label{Rie-eq-I}
\frac{1}{n}\sum^{n}_{k=1}\Gamma_{k+m} \leq \frac{n}{\operatorname{vol} (M^{n})} \int_{M^{n}} H^{2}dv.
\end{equation}\end{corr}
\begin{proof} We assume that $\{\mathfrak{s}_{i}\}_{i=1}^{m}$ is an orthonormal basis for the zero-eigenspace of $\slashed{D}^{2}$ on $C^{\infty}(\slashed{E})$. Under this assumption, it is evident that $\left\langle \mathfrak{s}_{l}, \mathfrak{s}_{l}\right\rangle$ is a constant function on $M^{n}$. Since $\left(\mathfrak{s}_{l}, \mathfrak{s}_{l}\right)=1,$ it is easy to verify that
\begin{equation*}\left\langle \mathfrak{s}_{l}, \mathfrak{s}_{l}\right\rangle=\frac{1}{\operatorname{vol}(M^{n})}.\end{equation*}  The Bochner-Lichnerowicz-Weitzenb\"{o}ck formula implies that $\slashed{\nabla} \mathfrak{s}=0$, i.e., $\mathfrak{s}$ is parallel. As a  consequence, this claim follows now from Theorem \ref{thm-main}, since \begin{equation*}\left(H^{2} \mathfrak{s}_{l}, \mathfrak{s}_{l}\right)=\int_{M^{n}} H^{2}\left\langle \mathfrak{s}_{l}, \mathfrak{s}_{l}\right\rangle dv,~\text{and}~ \Gamma_{l}=0,\end{equation*} for $1 \leq l \leq m$. Therefore, this finishes the proof.
\end{proof}

\begin{rem} In {\rm \cite{R}}, Reilly investigated the closed eigenvalue problem and established an upper bound as follows:
\begin{equation}\label{Rie-ineq}\lambda_{1}\leq\frac{n}{\operatorname{vol} (M^{n})} \int_{M^{n}} H^{2}dv,\end{equation}where $\lambda_{1}$ is the first nonzero eigenvalue of Laplace-Beltrami operator $d^{*} d+d d^{*}$ restricted to functions, i.e., $\slashed{E}=\mathbb{C}$. For this case, the $0$-eigenspace is one-dimensional. Furthermore, Reilly's bound was extended by Anghel to the analogy of Dirac operator $\slashed{D}^{2}$. We refer the reader to  {\rm \cite[Theorem 3.6]{An}} for details. To be specific, Anghel obtained eigenvalue inequality \eqref{An-H} in term of Dirac operator $\slashed{D}^{2}$.  According to the monotonically increasing arrangement of eigenvalues, it is easy to see that inequality \eqref{Rie-eq-I} deduces to inequality \eqref{An-H}. Therefore, our eigenvalue bound of general Reilly type actually generalizes Anghel's one. \end{rem}

\subsection{Eigenvalues on the Unit Sphere}
If the isometric embedding \begin{equation*}M^{n} \stackrel{ \sigma }{\rightarrow} \mathbb{S}^{n+p-1}(1) \stackrel{\jmath}{\rightarrow} \mathbb{R}^{n+p}\end{equation*} is the minimal embedding into the Euclidean unit sphere $\mathbb{S}^{n+p-1} \subset \mathbb{R}^{n+p}$, where $\jmath$ denotes the including embedding map, then the coordinate functions $x_{1}, x_{2}, \ldots, x_{n+p}$ are the eigenfunctions of the Laplacian with the eigenvalue $n$. Thus in the case of minimally embedded submanifolds into Euclidean spheres, one gets an inequality of Levitin-Parnovski type.

\begin{thm}\label{thm-sphere1}Let $M^{n}$ be a compact Riemannian spin manifold of dimension $n$ and \begin{equation*}\sigma=\left(x_{1}, \ldots, x_{n+p}\right): M^{n} \stackrel{\sigma}{\rightarrow} \mathbb{S}^{n+p-1}(1) \stackrel{\jmath}{\rightarrow} \mathbb{R}^{n+p}\end{equation*} be the isometrical embedding into the Euclidean unit sphere $\mathbb{S}^{n+p-1} \subset \mathbb{R}^{n+p}$ with mean curvature $\bar{H}$, where $\jmath$ stands for the including map from the unit sphere $\mathbb{S}^{n+1}(1)$ the Euclidean space $\mathbb{R}^{n+p}$. Let $\slashed{D}$ be the Dirac operator of spinor bundle $\slashed{E}$ over $M^{n}$ and let $\left\{\Gamma_{j} ; \mathfrak{s}_{j}\right\}_{j \in \mathbb{N}}$ be a spectral resolution of $\slashed{D}^{2}$. Then one gets   the inequality of lower order eigenvalues
\begin{equation}\label{equation-sphere-1}
\sum^{n}_{k=1}\Gamma_{j+k} \leq(n+4)\Gamma_{j}+n^{2}\int_{M^{n}}\left(\bar{H}^{2}+1\right)\langle\mathfrak{s}_{j}, \mathfrak{s}_{j}\rangle dv-4\int_{M^{n}}\langle\Re \mathfrak{s}_{j}, \mathfrak{s}_{j}\rangle dv.
\end{equation}

\end{thm}
\begin{proof}Since the unit sphere can be canonically embedded into the Euclidean space, we have the following diagram:

\begin{equation*}\begin{aligned}
\xymatrix{
  M^{n}\ar[dr]_{\jmath\circ \sigma} \ar[r]^{\sigma}
                & \mathbb{S}^{n+p-1}(1) \ar[d]^{\jmath}  \\
                &\mathbb{R}^{n+p}             }\end{aligned} \end{equation*}
where $\jmath: \mathbb{S}^{n+p-1}(1) \rightarrow \mathbb{R}^{n+p}$ is the canonical embedding from the unit sphere $S^{n+p-1}(1)$ into $\mathbb{R}^{n+p}$, and $\sigma: M^{n} \rightarrow \mathbb{S}^{n+p-1}(1)$ is an isometric immersion. Clearly, composition mapping $\jmath \circ \sigma: M^{n} \rightarrow \mathbb{R}^{n+p}$ is an isometric immersion from $M^{n}$ to $\mathbb{R}^{n+p}$. Let $\bar{H}$ and $H$ be the mean curvature of $\sigma$ and $\jmath \circ \sigma$, respectively. Then, we have
\begin{equation}\label{H=H-tilde}
H^2=\bar{H}^2+1 .
\end{equation}
Applying Theorem \ref{thm-main} directly, we can get \eqref{equation-sphere-1}. Therefore, we finish the proof of Theorem \ref{thm-sphere1}.\end{proof}

\begin{rem}Denoting $v_{i}$ by\begin{equation*}v_{i}=\Gamma_{i}+\frac{1}{4}\inf _{\sigma \in \tilde{\Pi}} \left(\max _{\Omega} n^2 \left(\bar{H}^2+1\right)-4\min _{M^{n}} S\right),\end{equation*}then we have
\begin{equation} \label{sphere-eq1}
\sum_{i=1}^{n}\left(v_{i+j}-v_{j}\right) \leq 4 v_{j},
\end{equation}
where $S$ is scalar curvature of $M^{n}$ and $\tilde{\Pi}$ denotes the set of all the isometric embedding maps from $M^{n}$ to the unit sphere $\mathbb{S}^{n+1}(1)$.
In particular, if the isometrical embedding $\sigma: M^{n} \rightarrow \mathbb{S}^{n+p-1}(1)$ is minimal, equation \eqref{H=H-tilde} implies that $v_{i}$ in \eqref{sphere-eq1} is given by
 \begin{equation*}v_{i}=\Gamma_{i}+\frac{1}{4}\inf _{\sigma \in \tilde{\Pi}} \left(n^{2}-4\min _{M^{n}}S\right).\end{equation*}\end{rem}
\begin{corr}[\textbf{General Reilly type formula II}] \label{Rie-for-II}In addition to the hypotheses of Theorem {\rm \ref{thm-sphere1}} assume that zero is in the spectrum of $\slashed{D}^{2}$ on $C^{\infty}(\slashed{E})$, the dimension ${\rm dim}(\slashed{E}_{0})=m$, where
$\slashed{E}_{0}$ stands for the zero-eigenspace of $\slashed{D}^{2}$ on $C^{\infty}(\slashed{E})$, and $\Re|_{ \slashed{E}}=0$.  Then the first nonzero $n$ eigenvalues of $\slashed{D}^{2}$ on $C^{\infty}(\slashed{E})$ satisfy
\begin{equation}\label{Rie-eq-II}
\frac{1}{n}\sum^{n}_{k=1}\Gamma_{k+m} \leq   \frac{n}{\operatorname{vol}(M^{n})} \int_{M^{n}}\left(\bar{H}^2+1\right)dv.
\end{equation}\end{corr}
\begin{proof}  Noticing \eqref{H=H-tilde}, we can conclude this theorem by the same method as the proof of Corollary \ref{Rie-ty-for-1}. Here, we omit it.

\end{proof}
\begin{rem} B\"{a}r  {\rm \cite{Ba2}} considered the eigenvalues of Dirac operators in term of twisted Clifford bundle on some hypersurfaces of $(n+1)$-dimensional $\mathbb{R}^{n+1}$ and unit sphere $\mathbb{S}^{n+1}(1)$ and proved two Rielly type bounds \eqref{Bar-in1} and \eqref{Bar-in2}. Therefore, our upper bounds  of general Rielly type \eqref{Rie-eq-I} and \eqref{Rie-eq-II}  can be compared with  \eqref{Bar-in1} and \eqref{Bar-in2}, respectively.
\end{rem}
By Theorem \ref{thm-sphere1}, one can establish a universal bounds as follows.
\begin{corr}[\textbf{Universal bound}]  In addition to the hypotheses of Theorem {\rm \ref{thm-sphere1}}, we suppose that  $M^{n}$ is an $n$-dimensional minimal immersed hypersurface of the unit $\mathbb{S}^{n+1}(1)$  and $\mathfrak{R}|_{\slashed{E}} \geq c_{3}Id$, for some $c_{3} \in \mathbb{R}$.
 Then the first $n$ nonzero eigenvalues of $\slashed{D}^{2}$ acting on $C^{\infty}(\slashed{E})$ satisfy the following universal bound:

\begin{equation}\label{Rie-eq-II}
\sum^{n}_{k=1}\Gamma_{j+k} \leq(n+4)\Gamma_{j}+n^{2} -4c_{3}.
\end{equation}

\end{corr}

\begin{rem}Since inequality \eqref{Rie-eq-II} is independent of the manifold  $M^{n}$ but only dependent of the dimension of $M^{n}$, it is a universal bound.\end{rem}
The remainder of this subsection, we would like to discuss the case of surfaces isometrically immersed into $\mathbb{S}^{3}(1)$. On one hand,
the famous Willmore conjecture states as follows:
\begin{Will}[1965,\cite{Will}]The integral of the square of the mean curvature of a torus immersed in $\mathbb{R}^3$ is at least $2 \pi^2$.\end{Will}

Recall that the Willmore energy of a closed surface $M^{2} \subset \mathbb{S}^{3}(1)$ is defined by the following quantity
\begin{equation}\label{W-energy}
\mathcal{W}(M^{2})=\int_{M^{2}}\left(1+\bar{H}^{2}\right) dv.
\end{equation} We note that Willmore energy is  a conformal invariant.
By setting up a profound min-max theory, Marques and Neves proved the following theorem\cite[Theorem A]{MN}:

\begin{thm}[\textbf{2014, Marques and Neves}]\label{thm-MN}Let $M^{2} \subset \mathbb{S}^{3}(1)$ be an embedded closed surface of genus $\mathfrak{g}_{0}\geq 1$. Then
\begin{equation}\label{will-conj-ineq}
\mathcal{W}(M^{2}) \geq 2 \pi^{2},
\end{equation}
and the equality holds if and only if $M^{2}$ is the Clifford torus up to conformal transformations of $\mathbb{S}^{3}(1)$.
\end{thm}
Li and Yau \cite[Theorem 6]{LY} proved that if an immersion $\psi: M^{2} \rightarrow \mathbb{S}^{3}(1)$ covers a point $x \in \mathbb{S}^{3}(1)$ at least $k$ times, then $\mathcal{W}(M^{2}) \geq 4 \pi k$. Also see \cite{MR}. Furthermore, combining with Li and Yau's result,  Theorem \ref{thm-MN} deduces the Willmore conjecture. Thus, Marques and Neves made an affirmative answer to Willmore's conjecture.

On the other hand, by refining an variation  technique of  Lichnerowicz formula, Hijazi \cite{Ba1,Ba2}   estimated
the smallest Dirac eigenvalue associated with the eigenvalue of Yamabe operator.  However, based on Sobolev embedding theorems, Lott \cite[Proposition 1]{Lott} and Ammann \cite{Am} to proved that for each closed spin manifold $M^{n}$ and each conformal class $\left[g_0\right]$ on $M^{n}$, there exists a constant $C$ depending on $M^{n}$ and the conformal class $\left[g_0\right]$ such that all nonzero Dirac eigenvalues $\bar{\Gamma}_{k}$ with respect to all Riemannian metrics $g \in\left[g_0\right]$ satisfy
$$
\bar{\Gamma}_{k} \geq \frac{C}{\operatorname{vol}(M^{n})^{2 / n}}.
$$
Next, let us consider the $2$-sphere $M^{2}=\mathbb{S}^2$. It is well known that there is only one conformal class of metrics in the sense of the action of the diffeomorphism group  and we therefore get a nontrivial lower bound for all metrics.  Under this situation, Lott conjectured that the optimal constant should be $C=4 \pi$. Returning to the Bochner technique B\"{a}r \cite[Theorem 2]{Ba1} established the following:

\begin{thm}Let $\bar{\Gamma}_{1}$ be any Dirac eigenvalue of the $2$-sphere $\mathbb{S}^2$ equipped with an arbitrary Riemannian metric. Then
$$
\bar{\Gamma}_{1} \geq \frac{4 \pi}{\operatorname{area}\left(\mathbb{S}^2\right)}
$$
Equality is attained if and only if $\mathbb{S}^2$ carries a metric of constant Gauss curvature.
\end{thm}
However, for the sum of the first two eigenvalues and based on the upper estimate, we propose a conjecture on the
lower bound of Dirac operator $\slashed{D}^{2}$ on a closed surface with one genu immersed in $\mathbb{S}^{3}(1)$. We note the following conjecture is closely related to the Willmore
conjecture.

\begin{con}\label{con-Zeng}Let $M^{2} \subset \mathbb{S}^{3}(1)$ be an embedded closed surface with one genu and $\Re|_{ \slashed{E}}=0$.  Then the first two nonzero  eigenvalues $\bar{\Gamma}_{1}$ of $\slashed{D}^{2}$ on $C^{\infty}(\slashed{E})$ satisfy
\begin{equation}\label{con-ineq-z} \frac{1}{2}\sum^{2}_{i=1}\bar{\Gamma}_{i}\geq\frac{4\pi^{2}}{\operatorname{area}(M^{2})}.\end{equation}\end{con}

\begin{rem}From \eqref{Rie-eq-II}, \eqref{W-energy} and \eqref{con-ineq-z}, one can infer that
\begin{equation*}\frac{4\pi^{2}}{\operatorname{area}(M^{2})}\leq\frac{1}{2}\sum^{2}_{i=1}\bar{\Gamma}_{i}\leq\frac{2}{\operatorname{area}(M^{2})} \int_{M^{2}}\left(\bar{H}^2+1\right)dv=\frac{2\mathcal{W}(M^{2})}{\operatorname{area}(M^{2})},\end{equation*} which implies that inequality \eqref{will-conj-ineq} holds. In other words, if Conjecture {\rm \ref{con-Zeng}} is true, we shall provide an alternative proof for Willmore Conjecture.  Moreover, it will be interesting that Dirac operator proposed from quantum field theory will lead some important and new applications in differential geometry, in particular for the theory of minimal surfaces, once our conjecture is proven to be true.

\end{rem}

\subsection{Eigenvalues on the Projective Manifolds}

Firstly, we recall some fundamental facts for submanifolds on the projective spaces, and refer the reader to \cite{Cby,Che} for more details. Let $\mathbb{F}$ denote the field $\mathbb{R}$ of real numbers,
the field $\mathbb{C}$ of complex numbers or the field $\mathbb{Q}$ of quaternions.
Evidently,  the inclusion relationships $\mathbb{R} \subset \mathbb{C} \subset \mathbb{Q}$ holds. For any
element $z\in\mathbb{F}$, we define $z$'s conjugate by the following way:
If $z=z_{0}+z_{1} i+z_{2} j+z_{3} k \in \mathbb{Q}$
with $z_{0}, z_{1}, z_{2}, z_{3} \in \mathbb{R},$ then
$
\bar{z}=z_{0}-z_{1} i-z_{3} j-z_{3} k.
$
If $z=a+\sqrt{-1}b\in\mathbb{C}$, where $a,b\in \mathbb{R}$, then $\bar{z}$ agrees with the canonical complex conjugate, i.e., $\bar{z}=a-\sqrt{-1}b$.
Let us denote by $\mathbb{F}P^{m}$ the $m$-dimensional real projective space if $\mathbb{F}= \mathbb{R}$, the complex projective space with real dimension $2 m$ if $\mathbb{F}= \mathbb{C}$, and the quaternionic projective space with real dimension $4 m$ if $\mathbb{F}= \mathbb{Q}$, respectively. For convenience, we
define integers $d_{\mathbb{F}}$ as follows:

\begin{equation}\label{df}
d_{\mathbb{F}}=\operatorname{dim}_{\mathbb{R}}\mathbb{F}=\left\{\begin{array}{ll}
1, & \text { if } \mathbb{F} = \mathbb{R}; \\
2, & \text { if } \mathbb{F} = \mathbb{C}; \\
4, & \text { if } \mathbb{F} = \mathbb{Q}.
\end{array}\right.
\end{equation} Since projective manifold $\mathbb{F}P^{m}$ carries a canonical metric such that Hopf fibration
$\iota: \mathbb{S}^{d_{\mathbb{F}} \cdot(m+1)-1} \subset \mathbb{F}^{m+1} \rightarrow \mathbb{F}P^{m}$ is a Riemannian submersion,
the sectional curvature of $\mathbb{R}P^{m}$ is $1$, the holomorphic sectional curvature is $4$ and the quaternion sectional  curvature is $4$. Suppose $\mathcal{P} _{m+1}(\mathbb{F})$ is composed of all  $(m+1)\times(m+1)$ matrices over the field $\mathbb{F}$, i.e.,
\begin{equation*}\mathcal{P} _{m+1}(\mathbb{F})=\left\{(a_{ij})_{(m+1)\times(m+1)}|a_{ij}\in\mathbb{F}\right\},\end{equation*} and let
\begin{equation*}\mathcal{H}_{m+1}(\mathbb{F})=\left\{P \in \mathcal{P} _{m+1}(\mathbb{F}) \mid P^{*}:=\bar{^{t} P}=P\right\}\end{equation*}  be the vector space of $(m+1) \times(m+1)$ Hermitian
matrices with coefficients in the field $\mathbb{F}$. Then, one can endow $\mathcal{H}_{m+1}( \mathbb{F})$ with an inner product as follows:
$
\langle P, Q\rangle=\frac{1}{2} \operatorname{tr}(PQ),
$
where tr $(\cdot)$ denotes the trace for the given $(m+1) \times(m+1)$ matrix. It is clear that the map $\iota: \mathbb{S} ^{d_{\mathbb{F}} \cdot(m+1)-1} \subset \mathbb{F} ^{m+1} \rightarrow$
$\mathcal{H}_{m+1}(\mathbb{F})$ given by
\begin{equation}\label{iota-immersion}
\iota(\boldsymbol{z})=\boldsymbol{z}\boldsymbol{z}^{\ast}=\left(\begin{array}{llll}
\left|z_{0}\right|^{2} & z_{0} \bar{z_{1}} & \cdots & z_{0} \bar{z_{m}} \\
z_{1} \bar{z_{0}} & \left|z_{1}\right|^{2} & \cdots & z_{1} \bar{z_{m}} \\
\cdots & \cdots & \cdots & \cdots \\
z_{m} \bar{z_{0}} & z_{m} \bar{z_{1}} & \cdots & \left|z_{m}\right|^{2}
\end{array}\right)
\end{equation}
induces an isometric embedding $\iota$ from $\mathbb{F}P^{m}$ into $\mathcal{H}_{m+1}( \mathbb{F})$ through the Hopf fibration, where
$\boldsymbol{z}=(z_{0},z_{1},\cdots,z_{m})\in\mathbb{S} ^{d_{\mathbb{F}} \cdot(m+1)-1}.$ Moreover, $\iota\left( \mathbb{F}P^{m}\right)$ is a minimal submanifold of the hypersphere $\mathbb{S} \left(\frac{I}{m+1}, \sqrt{\frac{m}{2(m+1)}}\right)$ of $\mathcal{H}_{m+1}( \mathbb{F})$ with radius $\sqrt{\frac{m}{2(m+1)}}$ and
center $\frac{I}{m+1}$, where $I$ denotes the identity matrix.
In addition, we need a lemma (see \cite[Lemma 6.3 in Chapter 4]{Cby}, or a proof of this lemma in \cite{T}) as follows:

\begin{lem} \label{lem-proj}Let $\varsigma: M^{n}  \rightarrow \mathbb{F} P^{\text {m }}$ be an isometric immersion and $\iota$  an induced isometric embedding from $\mathbb{F}P^{m}$ into $\mathcal{H}_{m+1}( \mathbb{F})$ given by \eqref{iota-immersion}. Let $\tilde{\boldsymbol{H}}$ and $\boldsymbol{H}$ be the mean curvature vector fields of the immersions $\varsigma$ and $\iota \circ \varsigma,$ respectively. Then, we have
\[
\left| \boldsymbol{H}\right|^{2}=|\tilde{\boldsymbol{H}}|^{2}+\frac{4(n+2)}{3 n}+\frac{2}{3 n^{2}} \sum_{i \neq j} K\left(e_{i}, e_{j}\right),
\]
where $\left\{e_{i}\right\}_{i=1}^{n}$ is a local orthonormal basis of $\bar{\Gamma}(T M^{n})$ and $K$ is the sectional curvature of $\mathbb{F}P^{m}$ expressed $b y$
\[
K\left(e_{i}, e_{j}\right)=\left\{\begin{array}{ll}
1, & \text { if } \mathbb{F} = \mathbb{R}; \\
1+3\left(e_{i} \cdot J e_{j}\right)^{2}, & \text { if } \mathbb{F} = \mathbb{C}; \\
1+\sum_{r=1}^{3} 3\left(e_{i} \cdot J_{r} e_{j}\right)^{2}, & \text { if } \mathbb{F} = \mathbb{Q},
\end{array}\right.
\]
where $J$ is the complex structure of $\mathbb{C}P^{m}$ and $J_{r}$ is the quaternionic structure of $\mathbb{Q}P ^{m}$.\end{lem}
\begin{equation*}\eqno\Box\end{equation*}

Furthermore, from Lemma \ref{lem-proj}, we can deduce that
\begin{equation*}
\left|\boldsymbol{H}\right|^{2}=\left\{\begin{array}{ll}
|\tilde{\boldsymbol{H}}|^{2}+\frac{2(n+1)}{2 n}, & \text { for } \mathbb{R} P^{m}; \\
|\tilde{\boldsymbol{H}}|^{2}+\frac{2(n+1)}{2 n}+\frac{2}{n^{2}} \sum_{i, j=1}^{n}\left(e_{i} \cdot J e_{j}\right)^{2} \leq|\tilde{\boldsymbol{H}}|^{2}+\frac{2(n+2)}{n}, & \text { for } \mathbb{C} P^{m}; \\
|\tilde{\boldsymbol{H}}|^{2}+\frac{2(n+1)}{2 n}+\frac{2}{n^{2}} \sum_{i, j=1}^{n} \sum_{r=1}^{3}\left(e_{i} \cdot J_{r} e_{j}\right)^{2} \leq|\tilde{\boldsymbol{H}}|^{2}+\frac{2(n+4)}{n}, & \text { for } \mathbb{Q} P^{m}.
\end{array}\right.
\end{equation*}Hence,  from the above equations, it is not difficult to verify the following inequality:
\begin{equation}\label{HH}
\left| \boldsymbol{H}\right|^{2} \leq|\tilde{\boldsymbol{H}}|^{2}+\frac{2\left(n+d_{\mathbb{F}}\right)}{n}.
\end{equation}
In particular, the equality in \eqref{HH} holds if and only if $M^{n}$ is a complex submanifold of $\mathbb{C}P^{m}$ (for the case $\mathbb{C}P^{m}$ ) while $n \equiv 0(\bmod 4)$ and $M^{n}$ is an invariant submanifold of $\mathbb{Q}P^{m}\left(\text { for the case } \mathbb{Q}P^{m}\right)$. For a compact complex spin manifold $M^{n}$ with a holomorphic isometric embedding into the complex projective space, Chen and Sun \cite{CS} obtained some extrinsic estimates from above  for eigenvalues with lower or higher orders of the Dirac operator, which depend on the data of an isometric embedding of $M^{n}$. Next, we shall prove an extrinsic estimate of Levitin-Parnovski type for the Dirac operator on the projective manifolds.

\begin{thm} \label{thm-proj} Let $M^{n}$ be the $n$-dimensional compact spin manifold and $\Psi: M^{n} \rightarrow  \mathbb{F}P^{m}$ be an isometric embedding with mean curvature $H$. Let $\slashed{D}$ be the Dirac operator of spinor bundle $\slashed{E}$ over $M^{n}$ and let $\left\{\Gamma_{j} ; \mathfrak{s}_{j}\right\}_{j \in \mathbb{N}}$ be a spectral resolution of $\slashed{D}^{2}$. Then one obtains
\begin{equation}\label{proj-ineq-1}
\begin{aligned}
 \sum_{i=1}^{n}\left(\Gamma_{i+j}-\Gamma_{j}\right) \leq 4\left(\Gamma_{j}+\frac{n}{2}(n+d_{\mathbb{F}})+\frac{1}{4}\inf_{\Psi\in\hat{\Pi}} \sup _{M^{n}}\left(n^{2}\tilde{H}^{2}-4S\right)\right),
\end{aligned}
\end{equation}
where $S$ is the scalar curvature of $M^{n}$ and $\hat{\Pi}$ denotes the set of all isometric embedding maps from $M^{n}$ to $\mathbb{F}P^{m}$.
Furthermore, if the imbedding $\Psi$ is minimal, then
\begin{equation}\label{proj-ineq-2}
\begin{aligned}\sum_{i=1}^{n}\left(\Gamma_{i+j}-\Gamma_{j}\right) \leq 4\left(\Gamma_{j}+\frac{n}{2}(n+d_{\mathbb{F}})- \inf _{M^{n}} S\right).
\end{aligned}
\end{equation}\end{thm}

\begin{proof}
It is well known that there exists a canonical embedding map \begin{equation*}\iota: \mathbb{F}P^{m} \rightarrow \mathcal{H}_{m+1}( \mathbb{F})\end{equation*} from $\mathbb{F} P^{m}( \mathbb{F} = \mathbb{R} , \mathbb{C} , \mathbb{Q} )$ to Euclidean space $\mathcal{H}_{m+1}( \mathbb{F} )$. Therefore, for compact manifold $M^{n}$ isometrically immersed into the projective space $\mathbb{F} P^{m},$ one has the following diagram:
\begin{equation*}\begin{aligned}
\xymatrix{
  M^{n}\ar[dr]_{\iota\circ \varsigma} \ar[r]^{\varsigma}
                & \mathbb{F}P^{m} \ar[d]^{\iota}  \\
                &\mathcal{H}_{m+1}(\mathbb{F})             }\end{aligned} \end{equation*}
 $\varsigma: M^{n} \rightarrow$
$\mathbb{F}P^{m}$ denotes  an isometric immersion from $M^{n}$ to $\mathbb{F}P^{m}$. Then, the composite map \begin{equation*}\iota \circ \varsigma:  M^{n} \rightarrow \mathcal{H}_{m+1}( \mathbb{F})\end{equation*} is an isometric immersion from $M^{n}$ to $\mathcal{H}_{m+1}( \mathbb{F})$. According to inequality \eqref{HH} and Theorem \ref{thm-main}, we can conclude \eqref{proj-ineq-1}  and \eqref{proj-ineq-2}. Hence, it completes the proof of Theorem \ref{thm-proj}.

\end{proof}
\begin{rem} If we consider that $\mathbb{F}$ is a complex domain, i.e., $\mathbb{F}=\mathbb{C}$, Theorem  {\rm \ref{thm-proj}} deduces one of Chen and Sun's eigenvalue inequalities. See {\rm \cite[Theorem 4.1]{CS}}.    \end{rem}

\begin{rem}Denoting $\eta_{i}$ by
\begin{equation*}\eta_{i}=\Gamma_{i}+\frac{n}{2}(n+d_{\mathbb{F}})+\frac{1}{4}\inf_{\Psi\in\hat{\Pi}} \sup _{M^{n}}\left(n^{2}\tilde{H}^{2}-4S\right),\end{equation*} we have
\begin{equation*}
\begin{aligned}
 \sum_{i=1}^{n}\left(\eta_{i+j}-\eta_{j}\right) \leq 4\eta_{j}.
\end{aligned}
\end{equation*}\end{rem}

\begin{corr}[\textbf{General Reilly type formula III}] \label{Rie-for-II}In addition to the hypotheses of Theorem {\rm \ref{thm-proj}} assume that zero is in the spectrum of $\slashed{D}^{2}$ on $C^{\infty}(\slashed{E})$, the dimension ${\rm dim}(\slashed{E}_{0})=m$, where
$\slashed{E}_{0}$ stands for the zero-eigenspace of $\slashed{D}^{2}$ on $C^{\infty}(\slashed{E})$, and $\Re|_{ \slashed{E}}=0$. Then the first nonzero $n$ eigenvalues of $\slashed{D}^{2}$ on $C^{\infty}(\slashed{E})$ satisfy
\begin{equation}\label{Rie-eq-III}
\frac{1}{n}\sum^{n}_{k=1}\Gamma_{k+m} \leq   \frac{n}{\operatorname{vol}(M^{n})} \int_{M^{n}}\left[\tilde{H}^2+\frac{2\left(n+d_{\mathbb{F}}\right)}{n}\right]dv.
\end{equation}\end{corr}
\begin{proof}  Noticing \eqref{HH}, we can conclude this theorem by the same method as the proof of Corollary \ref{Rie-ty-for-1}. Here, we omit it.

\end{proof}According to Theorem \ref{thm-proj}, we can obtain a universal bound as follows
\begin{corr}[\textbf{Universal bound}]  In addition to the hypotheses of  Theorem \ref{thm-proj}, we suppose that  $M^{n}$ is an $n$-dimensional minimal immersed submanifold of the unit $\mathbb{F}P^{m}$ and $\mathfrak{R}|_{\slashed{E}} \geq c_{4}Id$ for some $c_{4}\in\mathbb{R}$.
 Then the first $n$ nonzero eigenvalues of $\slashed{D}^{2}$ acting on $C^{\infty}(\slashed{E})$ satisfy
\begin{equation}\label{Rie-eq-III}
 \sum_{i=1}^{n}\left(\Gamma_{i+j}-\Gamma_{j}\right) \leq 4 \Gamma_{j}+2n(n+d_{\mathbb{F}})-4c_{4}.
\end{equation}\end{corr}
\begin{rem}Clearly, inequality \eqref{Rie-eq-III} is independent of the manifold  $M^{n}$ but only dependent of the dimension of $M^{n}$. Therefore, it is also a universal bound.\end{rem}

\subsection{Eigenvalues of Atiyah-Singer Laplacian on Spin Manifolds}

In this subsection, we focus on estimating for the eigenvalues of the classical Dirac operator(Atiyah-Singer Laplacian) on a compact Riemannian manifold with spin structure.

\begin{thm}\label{thm-spin} If $M^{n}$ is an $n$-dimensional, isometrically immersed spin submanifold of the $(n+p)$-dimensional Euclidean space $\mathbb{R}^{n+p}$ via any immersed map $\sigma$, then

\begin{equation}\label{LPtyineq}
\sum^{n}_{k=1}\Gamma_{j+k} \leq(n+4)\Gamma_{j}+\inf_{\sigma\in\Pi}\max_{M^{n}}|\mathcal{B}|^{2},
\end{equation}
where $\mathcal{B}$ is the second fundamental form of the immersion $\sigma$.
\end{thm}
\begin{proof} Under the assumption of this theorem, we have \begin{equation*}\Re=\frac{S}{4},\end{equation*} where $S$ is scalar curvature, and thus the estimate in Theorem \ref{thm-main} becomes

\begin{equation*}
\sum^{n}_{k=1}\Gamma_{j+k} \leq(n+4)\Gamma_{j}+\int_{M^{n}}(n^{2}H^{2}-S)\langle\mathfrak{s}_{j}, \mathfrak{s}_{j}\rangle dv.
\end{equation*}
Therefore, from the Gauss equation \eqref{Gauss-for} we get the desired estimate.\end{proof}
Next, we would like to prove a stronger  Results than  Anghel's one  \cite[Corollary 3.8]{An}, which can be compared with eigenvalue inequality \eqref{D-eq-sq} due to Bunm \cite{Bau}, which is a upper bound and dependent of the  square of the second fundamental form  $|\mathcal{B}|^{2}$.

\begin{corr}Let $M^{n}$ be an $n$-dimensional spin manifold with nonvanishing $\hat{A}$-genus and $n$ an even number. Then, eigenvalues of  the classical Dirac operator $\slashed{D}^{2}$ on $M^{n}$ satisfies
\begin{equation}\label{an-eq-cor}
\sum_{i=1}^{n}\bar{\Gamma}_{i} \leq  \inf_{\sigma\in\Pi}\max _{M^{n}}|\mathcal{B}|^{2},
\end{equation}
where $\mathcal{B}$ is the second fundamental form of any isometric immersion $\sigma$ from $M^{n}$ to Euclidean space $\mathbb{R}^{n+p}$.\end{corr}

\begin{proof}Recall that the celebrated Atiyah-Singer index theorem  states that,  for any elliptic linear differential operator,  the topological index agrees with
the analytical index  \cite{AS} (or see \cite[p. 256]{LM}). Hence, for the classical Dirac operator, Atiyah-Singer index theorem implies that a nonvanishing $\hat{A}$-genus ensures the existence of harmonic spinors. This is equivalent to say that, for a spin manifold with nonvanishing $\hat{A}$-genus, there exist $m (m\geq1)$ trivial  eigenspinors in term of  Atiyah-Singer Laplacian. Therefore, we can assume that the dimension of zero-eigenspace is $m$, i.e.,  \begin{equation*}\Gamma_{1}= \Gamma_{2}=\cdots=\Gamma_{m}=0.\end{equation*} Thus, under this assumption, we can write $\Gamma_{m+i}$  as $\bar{\Gamma}_{i}$. Equivalently, \begin{equation*}\Gamma_{m+i}=\bar{\Gamma}_{i},\end{equation*} where $i=1,2,\cdots,n$. Therefore, the corollary follows from Theorem \ref{thm-spin}.  \end{proof}

\begin{rem} Anghel's argument is based on the following estimate for the two consecutive eigenvalues, while our technique is based on an upper bound of Levitin-Parnovski type.\end{rem}

\end{document}